\documentclass[amsthm]{elsart}

\usepackage{yjsco}

\usepackage{multirow} 
\usepackage{enumitem} 
\usepackage{xspace} 
\usepackage{listings} 

\usepackage{amsmath,amssymb}
\usepackage{amsthm}
\usepackage{bm}
\usepackage{pdfsync}
\usepackage{subfigure}
\usepackage[dvipsnames]{xcolor}
\usepackage[english]{babel}

\usepackage[T1]{fontenc}
\usepackage{yfonts}

\renewcommand{\theenumi}{\alph{enumi}}  

\theoremstyle{plain}
\newtheorem{theorem}{Theorem}[section]
\newtheorem{proposition}[theorem]{Proposition}
\newtheorem{lemma}[theorem]{Lemma}
\newtheorem{corollary}[theorem]{Corollary}

\theoremstyle{definition}
\newtheorem{example}[theorem]{Example}
\newtheorem{definition}[theorem]{Definition}
\newtheorem{remark}[theorem]{Remark}
\newtheorem{conjecture}[theorem]{Conjecture}
\newtheorem{algorithm}[theorem]{Algorithm}
\newtheorem{procedure}[theorem]{Procedure}

\newtheorem{problem}[theorem]{Problem}

\newcommand\MyAlgorithm[1]{
 \goodbreak\bigskip
 \hrule
 \begin{algorithm}\label{alg:#1} 
 \textsc{#1}
 \end{algorithm}
 \hrule\medskip
}

\newcommand\MyProcedure[1]{
 \goodbreak\bigskip
 \hrule
 \begin{procedure}\label{alg:#1} 
 \textsc{#1}
 \end{procedure}
 \hrule\medskip
}

\newcommand{\col}[2]{{#1}\,{:}\,{#2}}  
\newcommand{\sat}[2]{\col{#1}{#2^\infty}}
\newcommand{\Sat}[1]{\mathop{\rm Sat}\nolimits_{#1}}
\newcommand{\wSat}[1]{\mathop{\rm wSat}\nolimits_{#1}}

\newcommand{\ideal}[1]{\langle #1 \rangle}
\newcommand{\define}[1]{\textbf{\boldmath #1}}

\let\phi=\varphi
\let\rho=\varrho
\let\theta=\vartheta
\let\epsilon=\varepsilon

\newcommand\g{{\mathbf g}}
\newcommand\LC{\mathop{\rm LC}\nolimits}
\newcommand\LT{\mathop{\rm LT}\nolimits}

\newcommand\LTs{\mathop{\rm LT_\sigma}\nolimits}
\newcommand\LM{\mathop{\rm LM}\nolimits}
\newcommand\NF{\mathop{\rm NF}\nolimits}
\newcommand\Mat{\mathop{\rm Mat}\nolimits}
\newcommand\SB{\mathop{\rm SB}\nolimits}

\newcommand\ev{\mathop{\rm ev}\nolimits}
\newcommand\pig{\mathop{\pi_\g}\nolimits}
\def\Supp{\mathop{\rm Supp}\nolimits}

\newcommand\Rel{\mathop{\rm Rel}\nolimits}
\newcommand\Relg{\mathop{\rm Rel}_\g\nolimits}

\newcommand\saremLT{\mathop{\mathcal{SR}_{\LT}}\nolimits}
\newcommand\sarem{\mathcal{SR}}

\newcommand\SLTrem{$\mathcal{S}_{\LT}$-remainder\xspace}
\newcommand\Srem{$\mathcal{S}$-remainder\xspace}
\newcommand\satinterr{\mathrm{Sat}\mathcal{SI}}

\newcommand \ie {\textit{i.e.}}

\newcommand \NN {{\mathbb N}}
\newcommand \QQ {{\mathbb Q}}
\newcommand \CC {{\mathbb C}}
\newcommand \TT {{\mathbb T}}
\newcommand \ZZ {{\mathbb Z}}

\let\implies=\Rightarrow

\newcommand\sagbi{SAGBI basis\xspace}
\newcommand\sagbis{SAGBI bases\xspace}
\newcommand\satsagbi{SatSAGBI basis\xspace}

\newcommand \cocoa{\mbox{\rm
C\kern-.13em o\kern-.07 em C\kern-.13em o\kern-.15em A}}

\begin{document}

\begin{frontmatter}

\title{Saturations of Subalgebras, SAGBI Bases, and U-invariants}
\author{Anna Maria Bigatti}
\address{\scriptsize Dip. di Matematica,
\ Universit\`a degli Studi di Genova, \ Via
Dodecaneso 35,\
I-16146\ Genova, Italy}
\ead{bigatti@dima.unige.it}

\author{Lorenzo Robbiano}
\address{\scriptsize Dip. di Matematica,
\ Universit\`a degli Studi di Genova, \ Via
Dodecaneso 35,\
I-16146\ Genova, Italy}
\ead{robbiano@dima.unige.it}

\thanks{This research was partly supported  by the 
``National Group for Algebraic and Geometric Structures, and their 
Applications'' (GNSAGA-INdAM).}

\begin{abstract}
Given a polynomial ring~$P$ over a field~$K$, an element~$g \in P$, and
a $K$-subalgebra~$S$ of~$P$, we deal with the problem of saturating~$S$
with respect to~$g$,  i.e.\ computing $\Sat{g}(S) = S[g, g^{-1}]\cap P$. 
In  the general case we
describe a procedure/algorithm to compute a set of generators for $\Sat{g}(S)$
which terminates if and only if it is finitely generated.
Then we consider the more interesting case when $S$ is graded.
In particular, if $S$ is graded by a positive matrix~$W$ and  $g$ is an indeterminate, 
we show that 
if we choose a term ordering $\sigma$ of $g$-{\tt DegRev} type compatible with~$W$, 
then the two operations of computing a $\sigma$-\sagbi of $S$ and saturating $S$ with respect to~$g$ commute.
This fact opens the doors to nice algorithms for the computation of~$\Sat{g}(S)$.
In particular, under special assumptions on the grading
one can use the truncation of a $\sigma$-\sagbi 
and get the desired result. Notably, this technique can be applied to the problem of
directly computing some $U$-invariants, classically called semi-invariants, 
even in the case that $K$ is not the 
field of complex numbers.
\end{abstract}

\date{\today}


\end{frontmatter}

\section{Introduction}
\label{Introduction}
This paper has two main ancestors. Our attention to the problem 
discussed here was drawn by a nice discussion with Claudio Procesi 
about the paper~\cite{KP} where the following claim is made: 
\textit{If we want to understand $U$-invariants  
from these formulas  it is necessary to compute the intersection 
$S_n =\CC[c_2,\dots, c_n][a_0,a_0^{-1}]\cap \CC[a_0,\dots, a_n]$}.
Here~$\CC$ denotes the field of complex numbers, $a_0, \dots, a_n$ 
are indeterminates, and the formulas are expressions of the~$c_i$ given, 
for $i=1,\dots,6$,  as follows:

$$
\begin{array}{ccl}
c_2  &= & -a_1^{[2]}  +a_0a_2     \cr
c_3  &=&  2a_1^{[3]}  -a_0a_1a_2  +a_0^2a_3     \cr
c_4 &=&  -3a_1^{[4]}  +a_0a_1^{[2]}a_2  -a_0^2a_1a_3  +a_0^3a_4   \cr
c_5 &=&   4a_1^{[5]}  -a_0a_1^{[3]}a_2  +a_0^2a_1^{[2]} a_3  
-a_0^3a_1a_4  + a_0^4a_5\cr
c_6 &=& -5a_1^{[6]}+ a_0a_1^{[4]}a_2 -a_0^2a_1^{[3]}a_3 + a_0^3a_1^{[2]}a_4 
-a_0^4a_1a_5 +a_0^5a_6

\end{array}
$$
where $\alpha^{[i]}$ means $\tfrac{1}{i!}\alpha^i$. 
In version 1 of~\cite{KP}  the theoretical background for this claim
was fully explained and  in its Section~3.5 a sketch of an algorithm to compute 
$S_n$ was illustrated.
In version 3 of~\cite{KP} the authors dropped the section about the algorithm and 
wrote: 
\textit{A general algorithm for these types of problems has been in fact developed by 
Bigatti-Robbiano in a recent preprint},
referring to the first arXiv version of this paper.

\smallskip
 Why are the elements of $S_n$ called $U$-invariants?
A detailed explanation can be found in~\cite{KP}. For the sake of completeness,
let us summarise it here.

Let $\CC[x]_{\le n}$ denote the vector space of polynomials in $\CC[x]$ of degree $\le n$.
The algebra~$S_n$ of $U$-invariants of polynomials of degree $n$, is the subalgebra of the 
algebra of polynomial functions on $\CC[x]_{\le n}$  which are invariant under 
the action of $(\CC,+)$ 
defined by $p(x) \to p(x + \lambda)$  for $\lambda \in \CC$.
Now let ${\mathbb H}\CC[x,y]_n$ denote the vector space of homogeneous 
polynomials in $\CC[x, y]$ of degree~$n$,
and let $U = \{ (\begin{smallmatrix} 1 & \lambda \\ 0 & 1 \end{smallmatrix}) \}$, 
the unipotent subgroup of ${\rm SL}(2,\CC)$.
We can identify $\CC[x]_{\le n}$  with ${\mathbb H}\CC[x,y]_n$, and then the action 
of $(\CC,+)$ can be identified with the 
action of $U$ on the algebra of polynomial functions on~${\mathbb H}\CC[x,y]_n$.

For example, let $f(x) = a_0\tfrac{x^2}{2} +a_1x+a_2\in \CC[x]_{\le2}$
and compute $f(x+\lambda)$.
$$f(x+\lambda) \;=\; a_0\tfrac{(x+\lambda)^2}{2} +a_1(x+\lambda) +a_2 
\;=\;
a_0\tfrac{x^2}{2} +(a_0\lambda+a_1)x +
(a_0\tfrac{\lambda^2}{2}+a_1\lambda +a_0)
\,.$$
Then consider $c_2  = -a_1^{[2]}  +a_0a_2 = -\frac{a_1^2}{2}  +a_0a_2 $ (as
above), and compute the new $c_2$ relative to 
the coefficients of $f(x+\lambda)$.  We get
$-\tfrac{(a_0\lambda+a_1)^2}{2}
+a_0(a_0\tfrac{\lambda^2}{2}+a_1\lambda +a_2)$ 
which is equal to $c_2$ for every $\lambda\in \CC$, proving that $c_2$ is a $U$-invariant.

\medskip
The first motivation for our investigation is that
the problem of computing a set of generators of 
$S_n=\CC[c_2,\dots, c_n][a_0,a_0^{-1}]\cap \CC[a_0,\dots, a_n]$
can be viewed as a special case of the following task.

\begin{problem}\label{problem}
\textit{Given a field~$K$, a polynomial ring 
$P=K[a_0,a_1, \dots, a_n]$, and polynomials ${g_1,\dots, g_r \in P}$, let~$S$ 
denote the subalgebra $K[g_1, \dots, g_r]$ of~$P$, and let 
$g \in P{\setminus} \{0\}$. 
The problem is to  compute generators of the $K$-algebra 
$S[g, g^{-1}]\cap P$.}
\end{problem}

\smallskip

The second motivation for taking on this challenge is  the evidence of the 
analogy with the standard problem in computer algebra of computing the 
saturation of an ideal. 
The analogy is clearly explained by recalling that the saturation 
of an ideal $I \subseteq P$
with respect to~$g$ is the ideal $IP[g^{-1}]\cap P$.
How to compute the saturation of an ideal with respect to an element
in~$P$
and also with respect  to another ideal is well-understood and its 
solutions are described in the literature (see for instance~\cite[Section 3.5.B]{KR1}
 and~\cite[Sections 4.3  and 4.4]{KR2}) 
and implemented in most computer algebra systems.

On the other hand, the main problem formulated above 
has not received the same attention.
In this paper we describe a solution
if the algebra $\Sat{g}(S)= S[g, g^{-1}]\cap P$ is finitely generated.
As suggested by Gregor Kemper, a similar description is contained
in~\cite[Semi-algorithm 4.10.16]{DK2015}.

Then we present algorithms  removing redundant generators of $\Sat{g}(S)$ .
A strategy is to use elimination techniques, another strategy is to
make a good use of \sagbis. The acronym SAGBI 
stands for ``Subalgebra Analog to Gr\"obner Bases for Ideals''.
The theory of \sagbis was introduced by Robbiano and Sweedler in~\cite{RSOld}
and independently by Kapur and Madlener in~\cite{KaMa}.
Since then many 
improvements and applications were discovered (see for instance~\cite{CHV}).
A more modern approach is contained
in~\cite[Section 6.6]{KR2}, and~\cite[Chapter 11]{Stu96}, and a nice survey is 
described in~\cite{Bra}.
In~\cite{StTs} there are  results somehow related to this paper.

In the case $\Sat{g}(S)$ is not finitely generated, 
the algorithms turns out to be merely procedures
providing a sequence of algebras ever closer to $\Sat{g}(S)$.
 We show that this phenomenon can happen 
(see Examples~\ref{ex:nonsatreduce} and~\ref{ex:notfinite}),
which is not a surprise, since there are even examples of 
finitely generated subalgebras of a polynomial ring whose 
intersection is not finitely generated (see for instance~\cite{Mondal2017}).

We observe that our solutions do not require any assumption about the 
base field~$K$, and do not need that 
the polynomials $g_1, \dots g_r$ are homogeneous.
However, if they are homogeneous we have better results 
in Sections~\ref{sec:The Graded Case: Introduction}, \ref{sec:Minimalization}, 
\ref{sec:The Graded Case - Saturation}, 
which make a clever use of \sagbis and are the 
core of our paper.

\medskip
Now we give a  more precise description of the content of the paper.
The general setting is as follows.
We are given a field~$K$, a polynomial ring~$P$ over $K$,
a $K$-sub\-algebra~$S$ of~$P$,
and an element $g \in P{\setminus}\{0\}$.

In Section~\ref{sec:Basic Results}  we introduce the notion 
of the saturation $\Sat{g}(S)$ of  $S$ with respect to $g$. 
The main point is that if $g \in S$, then $\Sat{g}(S)=S:g^\infty$ 
(see Definition~\ref{def:ialgebrasatur}), as shown in 
Proposition~\ref{prop:MainProblem}. Using this fact we can rephrase 
the main problem addressed in this paper (see Problem~\ref{problem}).

Section~\ref{sec:The General Case} provides a first solution.
After recalling standard results in computer algebra 
(see Propositions~\ref{prop:subalgebraRepr}
and~\ref{prop:subalgebraMember}) we prove Theorem~\ref{thm:main} 
which shows how to add new elements to $S$ in order to get closer to  $\Sat{g}(S)$.
With the help of this result we prove Theorem~\ref{thm:inclusion} 
and Corollary~\ref{cor:sequence}. They provide the building 
blocks for Algorithm~\ref{alg:SubalgebraSaturation} which solves the problem if  $\Sat{g}(S)$ 
is finitely generated.  If not, the algorithm does not terminate
producing an infinite sequence of subalgebras ever closer to $\Sat{g}(S)$.
A case of $\Sat{g}(S)$ not being a finitely generated  $K$-algebra is shown in
Example~\ref{ex:nonsatreduce}.

Algorithm~\ref{alg:SubalgebraSaturation} 
is largely inspired by the suggestion contained in~\cite{KP} and similar 
to~\cite[Semi-algorithm 4.10.16]{DK2015}.

At this point we describe methods for rewriting the computed generators of $\Sat{g}(S)$.
The first part of Section~\ref{sec:Subalgebra-reduction} recalls different procedures to reduce 
elements in a subalgebra and the second part generalises these techniques to combine
reduction and saturation. 

As anticipated, our problem shows different features when the subalgebra~$S$ is graded.
Section~\ref{sec:The Graded Case: Introduction} marks the beginning of the most 
original part of the paper by showing some good aspects of 
this setting. Nevertheless, even in the graded case there are examples of subalgebras
whose saturation is not finitely generated (see Example~\ref{ex:notfinite}).

After the first glimpse provided in the previous section into the theory of \sagbis, 
its full strength comes alive in the case of a graded subalgebra.

The first benefit from the assumption that our subalgebra $S$ is positively 
graded is described in Section~\ref{sec:Minimalization} where we show how
to make a good use of a truncated \sagbi for minimalizing the generators of $S$
(see Algorithm~\ref{alg:SubalgebraMinGens}). 

Then we come to the  main novelties contained in Section~\ref{sec:The Graded Case - Saturation}.
Given a grading defined by a positive matrix $W$ and an indeterminate, say $a_0$, there are
special term orderings called of $a_0$-{\tt DegRev} type compatible with $W$ 
(see Definition~\ref{def:degrevtype}). If $\sigma$  is one of these, its full power is shown in Theorem~\ref{thm:SatandSAGBI}
which essentially says that the two operations of  
computing a $\sigma$-\sagbi of $S$ and saturating~$S$ commute. 
Using this fact, if the subalgebra~$S$ 
has a finite $\sigma$-\sagbi, then the problem of saturating $S$ with respect to $a_0$ is 
\textit{essentially} solved, and not only we get  a set of generators of $\Sat{a_0}(S)$ but 
also its $\sigma$-\sagbi.
Procedure~\ref{alg:SatSAGBI} captures this idea, and some examples show its good behaviour (see for instance  Examples~\ref{ex:standardgraded}
and~\ref{ex:standardgradedMedium}).
However, the reason why we said \textit{essentially} is that
 currently  we can only conjecture that the procedure is indeed an algorithm,
\ie~terminates, whenever $\Sat{a_0}(S)$ is finitely generated.

Finally, in Section~\ref{sec:U-invariants} we come back to the beginning of the story and 
use our methods to compute $U$-invariants via the computation of the algebras $S_n$. 
The ideas developed in 
Section~\ref{sec:The Graded Case - Saturation} frequently collide with the fact that 
computing a \sagbi can be very expensive. In many cases it is even not clear if it is finite or not.
So, what about computing a truncated \sagbi, as described in Section~\ref{sec:Minimalization}?
The problem is that the saturation of a polynomial with respect to $a_0$ lowers its degree unless
$\deg(a_0) = 0$, and unfortunately this condition cannot be paired with a term ordering of 
$a_0$-{\tt DegRev} type.  However, if the subalgebra $S$ is graded also with respect to another 
grading with $\deg(a_0) > 0$, we are in business. And this is the case of $U$-invariants.
Given a multi-grading of this special type and a term ordering of \hbox{$a_0$-{\tt DegRev}} 
type compatible with it we have the nice Algorithm~\ref{alg:TruncSatSAGBI}.
The algebras $S_3$ and~$S_4$ can be easily computed, and
indeed we compute even a \sagbi of them
together with a minimal  set of generators. 
But when we come to $S_5$ and~$S_6$ we need a bit of extra information 
which comes from the classical work, namely that the maximum weighted degree is 45 for both.
The weighted degree is such that $\deg(a_0) = 0$, so it suffices to use
Algorithm~\ref{alg:TruncSatSAGBI}, truncating the computation in weighted degree 45.
Once the truncated \sagbi is computed, we can use Algorithm~\ref{alg:SubalgebraMinGens} 
to get a minimal set of generators. To see some information about the computation 
of $S_5$ and $S_6$ see Examples~\ref{ex:U5} and \ref{ex:U6}.

\medskip
All the examples described in the paper were computed 
on a MacBook Pro~2.9GHz Intel Core~i7, using our implementation
in \cocoa\,5.
Unless explicitly stated otherwise,
we use the definitions and notation given in~\cite{KR1} and~\cite{KR2}.



\section*{Acknowledgements}
Thanks are due to Claudio Procesi who drew our attention 
to the problem. Moreover, the algorithm suggested in~\cite{KP} was  a
source of inspiration for our work.
Thanks are due to Hanspeter Kraft who took the job of 
checking that our results concerning $U_5$  (see Example~\ref{ex:U5})
are in agreement with the classical knowledge.
Special thanks are also due to Kraft and Procesi for modifying
  their paper by citing the preprint of this paper, as mentioned above.
We are grateful to Gregor Kemper for pointing out the 
result contained in~\cite[Semi-algorithm 4.10.16]{DK2015}.
Finally, we thank the referees for their careful reading and useful
suggestions.

\section{Basic Results}
\label{sec:Basic Results}

In this section we recall some basic definitions and results. 
In particular, we define the weak saturation and the saturation
of a subalgebra of~$P$  with respect to an element, 
which allows us to rewrite the main problem described in the introduction 
(see Problem~\ref{problem}).

In the following we  let~$K$ be a field, let $a_0,a_1,\dots, a_n$ be indeterminates, 
and let ${P = K[a_0,a_1, \dots, a_n]}$. 
The word \textit{term} is a synonym of \textit{power product}
while the word \textit{monomial} indicates  
a \textit{power product  multiplied by a coefficient}.
Consequently, if $\sigma$ is a term ordering and~$f$ is a polynomial,
the symbols $\LT_\sigma(f)$, $\LM_\sigma(f)$, $\LC_\sigma(f)$ denote the 
leading term, the leading monomial, and the leading coefficient of~$f$, 
so that we have $\LM_\sigma(f) = \LC_\sigma(f) \cdot \LT_\sigma(f) $.
For the ideal generated by  elements $g_1, \dots, g_r$ we use the notation $\ideal{g_1, \dots, g_r}$.

For a polynomial $f\in P$ we write \define{$\sat{f}{g}$} to denote
the \define{saturation of~$f$ with respect to~$g$}, i.e.\ 
the polynomial $f/g^i$, where $g^i$ is the highest 
power of~$g$ which divides~$f$.
 Given a subset~$T \subseteq P$,
 the $K$-subalgebra of~$P$ generated by $T$ is denoted by~\define{$K[T]$}.

We recall some definitions and properties 
from the context of ideals.
Let~$I$ be an ideal in~$P$, and $g\ne0$ in~$P$.
We first recall the \define{colon ideal}, defined as
$\col{I}{g} = \{f \in P \mid g\,f \in I\}$.
Notice that we naturally have that $\col{I}{g}$ is an ideal, and
$I \subseteq \col{I}{g}$.

Then, we recall
 the \define{saturation of~$I$ with respect to the element~$g$}, defined as
$\sat{I}{g} =  \bigcup_{i\in\NN} \col{I}{g^i}$,
and we also have $\sat{I}{g} = IP_g\cap P=IP[g^{-1}]\cap P$.

Next, we generalize the definitions above to the context of
  subalgebras, and we point out some properties which do not extend to
  this setting.

\begin{definition}\label{def:ialgebrasatur}
Let~$S$ be a $K$-subalgebra of the polynomial ring~$P$, and let
$g\ne0$ in~$P$.
\begin{enumerate}
\item  The subalgebra \define{$S[g^{-1}] \cap P$} is called the 
\define{weak saturation of~$S$ with respect to~$g$} and denoted
by \define{$\wSat{g}(S)$}.

\item  The subalgebra \define{$S[g, g^{-1}] \cap P$} is called the 
\define{saturation of~$S$ with respect to~$g$} and denoted
by \define{$\Sat{g}(S)$}.

\item We denote the
set $\{ f \in P \mid g^i\,f\in S\}$ by \define{$\col{S}{g^i}$}
and the set $\bigcup_{i \in \NN} \col{S}{g^i}$ by \define{$S:g^\infty$}.
\end{enumerate}
\end{definition}

\begin{remark}\label{rem:moresaturation}
Notice that $S\subseteq \col{S}{g}$  if and only if $g \in S$.
Thus, only in this case
$\col{S}{g^i}$ is an ascending chain of sets for increasing
$j\in\NN$.
We also observe that $S = \col{S}{g^0} \subseteq \sat{S}{g}$.
\end{remark}

The following example shows that in general $\col{S}{g}$ and $\sat{S}{g}$ are
  not subalgebras.

\begin{example}\label{rem:nosubalgebra}
Let $P=\QQ[a_0,a_1]$ and $S = \QQ[a_0a_1]\subseteq P$.
Trivially, $a_1$ is in $\col{S}{a_0}$, but
its square $a_1^2$ is not in~$\col{S}{a_0}$
 because $a_0a_1^2\not\in S$.

Now consider $S = \QQ[a_0^2a_1]\subseteq P$.
Then  $a_0a_1\in\col{S}{a_0}$, and $a_1\in\col{S}{a_0^2}$,
thus they are in $\sat{S}{a_0}$,
but their sum $a_0a_1+a_1$ is not in~$\sat{S}{a_0}$ because
$a_0^d(a_0a_1+a_1) \not\in S$ for any $d\in\NN$.
\end{example}


Next, we prove that if $g\in S$ then $\sat{S}{g}$ is a
  $K$-subalgebra of~$P$, and $\sat{S}{g}$
is indeed the saturation of~$S$ with
  respect  to~$g$.

\begin{proposition}\label{prop:MainProblem} 
Let~$S$ be a $K$-subalgebra of~$P$,
and $g\ne0$ in~$P$.
\begin{enumerate}
\item We have $\sat{S}{g} \subseteq \wSat{g}(S)$.

\item We have $\wSat{g}(\wSat{g}(S))=\wSat{g}(S)$.

\item If $A$ is a $K$-subalgebra of $P$ with
$S\subseteq A\subseteq \wSat{g}(S)$, then $\wSat{g}(A)=\wSat{g}(S)$.

\item \label{satequal} If $g \in S$ we have  $\sat{S}{g} = \wSat{g}(S)= \Sat{g}(S)$.
\end{enumerate}
\end{proposition}

\begin{proof}
To prove claim (a) we observe that for $f \in \sat{S}{g}$ 
there exists $r$ such that $g^r f \in S$ 
hence $f  = g^r f (g^{-1})^r  \in S[g^{-1}]\cap P$.

To prove claim (b) it is clearly enough to show 
$\wSat{g}(\wSat{g}(S))\subseteq\wSat{g}(S)$.
Let $f = \sum_{i=0}^d f_ig^{-i}$ with $f \in P$ and $f_i \in \wSat{g}(S)$ 
for $i=0,\dots, d$. Then we have the equalities 
$f_i = \sum_{j_i=0}^{\delta_i}s_{j_i}g^{-j_i}$
with $f_i \in P$ for $i=0,\dots,\delta_i$ and $s_{j_i} \in S$ for $i=0,\dots, d$, 
$j_i = 0,\dots, \delta_i$.  Hence we have 
$f =\sum_{\substack{i=0,\dots, d\\ j_i=0,\dots, \delta_i}} s_{j_i}g^{-i-j_i}$,
which shows that $f \in \wSat{g}(S)$.

Let us prove claim (c).
From the assumption $S\subseteq A\subseteq \wSat{g}(S)$ 
we get the chain of inclusions 
$\wSat{g}(S) \subseteq \wSat{g}(A) \subseteq \wSat{g}( \wSat{g}(S))$, and the 
conclusion follows from claim (b).

Finally,  we prove claim (d). The second equality is obvious, 
and from (a) we get the inclusion $\sat{S}{g} \subseteq \wSat{g}(S)$.
To conclude the proof, we need to show the
inclusion $\Sat{g}(S)\subseteq \sat{S}{g}$.
An element  $f \in S[g^{-1}]\cap P$  can be viewed as 
polynomial in $g^{-1}$ with coefficients~$s_i\in S$, hence it can be written as 
$f = \sum_{i=0}^d s_i g^{-i}=(\sum _{i=0}^d  s_ig^{d-i})/g^d$.
The assumption $g\in S$ implies that 
$f g^d = \sum _{i=0}^d  s_i g^{d-i} \in S$, hence $f \in \sat{S}{g}$.
\end{proof}

Example~\ref{rem:nosubalgebra} shows that without the assumption $g \in S$,
the set $\sat{S}{g}$ need not  be a $K$-algebra, hence  
 the inclusion in item (a) may be strict.

Under the light of these definitions and properties, we rephrase
the problem stated in the introduction (see~Problem~\ref{problem})
with the extra assumption that $g\in S\setminus\{0\}$.

\begin{problem}\label{problem2}
\textit{Given a field~$K$, a polynomial ring 
$P=K[a_0,a_1, \dots, a_n]$, and polynomials $g_1,\dots, g_r \in P$, let~$S$ 
denote the subalgebra $K[g_1, \dots, g_r]$ of~$P$, 
and let $g\in S\setminus\{0\}$.
The problem is to  compute a set of generators of $\sat{S}{g}$}.
\end{problem}

The assumption that $g$ is an element of $S$ can be weakened as shown 
by the following proposition.

\begin{proposition}\label{prop:Polya0InS}
Let $S$ be a $K$-subalgebra of $P$, let $g \in P{\setminus}\{0\}$,
and let 
$f(z)\in K[z]{\setminus}K$. 
If $f(g) \in S$  then
$\wSat{g}(S)=\Sat{g}(S[g]) = \sat{S[g]}{g}$.
\end{proposition}

\begin{proof}
The second equality trivially follows from 
Proposition~\ref{prop:MainProblem}.(\ref{satequal})
because $g \in S[g]$.
Let us prove $\wSat{g}(S)=\Sat{g}(S[g])$
using induction on $\deg(f(z))$.
If $\deg(f(z))=1$, \ie~$f(z) = c_1z+c_0$,
  clearly $g = (c_1)^{-1}\cdot(f(g)-c_0)\in S$,
hence the claim 
follows from Proposition~\ref{prop:MainProblem}.(\ref{satequal}).
Then assume that if a $K$-subalgebra $A$
of~$P$ contains $f(g)$ with $\deg(f(z))<d$ then 
$\wSat{g}(A)=\Sat{g}(A[g])$.

Now, we let $f(g) \in S$ with $\deg(f(z)) = d$, 
\ie~$f(z) =\sum_{i=0}^dc_iz^i$ with $c_d\ne0$,
and we let $\tilde{f}(z) = \sum_{i=1}^dc_iz^{i-1}$,
thus
$g\cdot\tilde{f}(g) = f(g){-}c_0 \in S$,
in other words, $\tilde{f}(g)\in S:g^\infty$.
Then, by Proposition~\ref{prop:MainProblem}.(a)
it follows that $\tilde{f}(g) \in \wSat{g}(S)$, therefore
 we define $A = S[\tilde{f}(g)]$
and we have $S\subseteq A \subseteq \wSat{g}(S)$.
Consequently, by Proposition~\ref{prop:MainProblem}.(c),
$$\wSat{g}(A) = \wSat{g}(S) \eqno{(*)}$$
On the other hand, from $\deg(\tilde{f}(z)) = d-1$ and the inductive assumption 
we get the equalities 
$\wSat{g}(A) = \Sat{g}(A[g])$. 
From the obvious equality
$A[g] =S[\tilde{f}(g)][g] = S[g]$ we get 
$\wSat{g}(A) = \Sat{g}(S[g])$ which, combined with $(*)$, concludes the proof.
\end{proof}


\section{The General Case}
\label{sec:The General Case}

In this  section we tackle Problem~\ref{problem2}. 
We start with the following theorem which shows how to add new 
generators to a subalgebra of~$P$ in order to get closer to its saturation.
The theorem uses generators of $\Relg(g_1, \dots, g_r)$ (see~Definition~\ref{def:IdealOfRelations})
which  can be effectively computed according to Proposition~\ref{prop:subalgebraRepr}.


\subsection{Ideal of Relations and Subalgebra Membership}
\label{sec:elim}

The following $K$-algebra homomorphisms will be used systematically throughout
the paper.
Let $S = K[g_1, \dots, g_r] \subseteq P=K[a_0,\dots, a_n]$ and
let $g \in P$.
We will use
the homomorphism
 \define{$\ev$}: $K[x_1,\dots,x_r] \longrightarrow P$,
defined by $\ev(x_i) = g_i$,
and  the canonical homomorphism
 \define{$\pi_g$}: $P \longrightarrow P/\ideal{g}$.
The fundamental notion of an ideal of relations is recalled.

\begin{definition}\label{def:IdealOfRelations}
The kernel of the composition $\pi_g\!\circ \ev$
is called the \define{ideal of relations of $g_1,\dots, g_r$  modulo $g$} and 
is denoted by~\define{$\Relg(g_1,\dots, g_r)$}.
\end{definition}

In the following proposition we show how to
 compute $\Rel_g(g_1, \dots, g_r)$ using an elimination ideal.
We recall the following
propositions
using new indeterminates $y_1,\dots,y_m$ to emphasize that they are 
quite general.

\begin{proposition}[\define{Computing $\Relg$}]
\label{prop:subalgebraRepr}\ \\
Let  $g,\, g_1, \dots, g_r \in K[y_1, \dots, y_m]$.
Then let $Q=K[x_1, \dots, x_r, y_1, \dots, y_m]$, 
and define the ideal
 $J = \ideal{g, \, x_1-g_1, \dots, x_r-g_r}\subseteq Q$.
\begin{enumerate}
\item We have the equality 
$\Rel_g(g_1,\dots,g_r) = J\cap K[x_1, \dots, x_r]$.

\item Let~$G$ be the reduced $\sigma$-Gr\"obner basis of $J$
where $\sigma$ is an elimination  ordering for $\{y_1, \dots, y_m\}$.
Then we have $\Rel_g(g_1,\dots,g_r) = \ideal{G\cap K[x_1,\dots, x_r]}$.
\end{enumerate}
\end{proposition}

\begin{proof}
See~\cite[Proposition 3.6.2]{KR1}.
\end{proof}

We will also need to test subalgebra membership.  A method for
checking it is recalled here.

\begin{proposition}[\define{Subalgebra Membership Test}]\
\label{prop:subalgebraMember}\\
Let  $g_1, \dots, g_r \in K[y_1, \dots, y_m]$.
Then let $Q=K[x_1, \dots, x_r, y_1, \dots, y_m]$, 
and define the ideal
 $J = \ideal{x_1-g_1, \dots, x_r-g_r}\subseteq Q$.
Finally, let $S = K[g_1,\dots, g_r] \subseteq K[y_1, \dots, y_m]$.

Then a polynomial $f \in K[y_1, \dots, y_m]$ is such that $f \in S$ if
and only if 
we have $\NF_{\sigma, J}(f) \in K[x_1,\dots, x_r]$
where $\sigma$ is an elimination ordering for $\{y_1,\dots, y_m\}$. 
In this case, if we 
let  $h=\NF_{\sigma, J}(f)$, then $ f = h(g_1,\dots, g_r)$ is an explicit 
representation of $f$ as an element of~$S$.
\end{proposition}

\begin{proof}
See~\cite[Corollary 3.6.7]{KR1}.
\end{proof}

We are ready to prove the first theorem of this paper.

\begin{theorem}\label{thm:main}
Let $g_1,\dots, g_r \in P = K[a_0,a_1,\dots,a_n]$, 
let $S = K[g_1, \dots, g_r]$,
and let $\g\in S{\setminus}\{0\}$.
Then  let $\{H_1, \dots, H_t\}$ be a set of gen\-erators of the ideal 
$\Relg(g_1, \dots, g_r)$, 
and finally let ${\tilde{h}_{i}}^{\mathstrut} \!=\! H_i(g_1, \dots, g_r)/\g$
and $h_{i}\!=\! \sat{H_i(g_1, \dots, g_r)}{\g}$ for $i= 1, \dots, t$. 
We have
$$ S \subseteq K[\col{S}{\g}] =
K[g_1, \dots, g_r, \tilde{h}_{1}, \dots, {\tilde{h}_{t}}^{\mathstrut}]
\subseteq  K[g_1, \dots, g_r, h_{1}, \dots, h_{t}] \subseteq \sat{S}{\g}$$
\end{theorem}

\begin{proof} 
The first inclusion follows from Remark~\ref{rem:moresaturation}.

Now we prove the inclusion
$K[\col{S}{\g}] \subseteq 
K[g_1, \dots, g_r, \tilde{h}_{1}, \dots, {\tilde{h}_{t}}^{\mathstrut}]$.
If $f \in \col{S}{\g}$ we deduce that $\g f \in
K[g_1, \dots, g_r]$, hence 
there  is a  polynomial 
$F \in K[x_1, \dots, x_r]$ such that we have $\g f= F(g_1, \dots, g_r)$.
This means that ${F(x_1, \dots, x_r) \in  \Relg(g_1, \dots, g_r)}$,
thus~$F$ may be written as $\sum_{j=1}^tB_j H_j$ which implies
$F(g_1, \dots, g_r)\!= \sum_{j=1}^tB_j(g_1, \dots, g_r) H_j(g_1, \dots, g_r)\!=
\sum_{j=1}^tB_j(g_1, \dots, g_r) \g\tilde{h}_{j} $,
and hence we have 
$\g f \!=\! F(g_1, \dots, g_r)\!=\!\sum_{j=1}^tB_j(g_1, \dots, g_r) \g\tilde{h}_{j} $.
From this relation we deduce  the equality
$f =  \sum_{j=1}^tB_j(g_1, \dots, g_r) \tilde{h}_{j}$.

Next we prove that
$K[g_1, \dots, g_r, \tilde{h}_{1}, \dots,
{\tilde{h}_{t}}^{\mathstrut}] \subseteq K[\col{S}{\g}]$.
Firstly, $\g\, g_i \in S$ for every $i =1,\dots, r$ since $\g \in S$.
Moreover, we have   $\g\tilde{h}_{i} = H_i(g_1,\dots, g_r) \in S$,  hence 
also $\tilde{h}_{i}  \in \col{S}{\g}$ for $i = 1, \dots t$.
Thus the inclusion is proved which concludes the proof of 
the equality  $ K[\col{S}{\g}] =
K[g_1, \dots, g_r, \tilde{h}_{1}, \dots, \tilde{h}_{t}]$.

The inclusion $K[g_1, \dots, g_r, \tilde{h}_{1}, \dots,
{\tilde{h}_{t}}^{\mathstrut}] \subseteq K[g_1, \dots, g_r, h_{1},
\dots, h_{t}]$ follows again from the assumption that $\g \in S$, and
the last inclusion of the claim is clear since $\sat{S}{\g}$ is a
$K$-algebra by Proposition~\ref{prop:MainProblem}.(b).
\end{proof}

The following  example shows that if $\g\notin S$ then 
$S \subseteq K[\col{S}{\g}]$ may not hold, and $K[\col{S}{\g}]$ may not be a 
finitely generated $K$-algebra.

\begin{example}\label{ex:a0notinS}
Let $P = K[a_0,a_1]$, and let $\g=a_0$.
\begin{enumerate}
\item Let $S=K[a_1]$. Then $\col{S}{a_0} = \{0\}$, hence $K[\col{S}{a_0}] = K$.

\item Let $S \!= K[a_0a_1]$. 
Then $K[\col{S}{a_0}] = K[a_1,a_0a_1^2,a_0^2a_1^3,\dots, a_0^ia_1^{i+1},\dots]
\subsetneq K[a_1,a_0a_1]$,
and $S\not \subset K[\col{S}{a_0}]$ since $a_0a_1 \notin K[\col{S}{a_0}]$.
\end{enumerate}

\end{example}

A straightforward consequence of Theorem~\ref{thm:main}  is 
an interesting independence of the set of generators of 
the ideal $\Relg(g_1, \dots, g_r)$.

\begin{corollary}\label{rem:notdepend}
With the same assumptions of the theorem,
let $\{H_1',\! \dots,\! H_u'\}$ 
be another set of generators of 
$\Relg(g_1,\! \dots, \! g_r)$, and 
let ${\widetilde{h'}_{i}}^{\mathstrut}= H_i'(g_1,\! \dots\! g_r)/a_0$  for $i=1,\dots, u$.
Then we have 
$K[g_1, \dots, g_r, \tilde{h}_{1}, \dots, \tilde{h}_{t}]  = 
K[g_1, \dots, g_r, {\widetilde{h'}_{1}}^{\mathstrut}, \dots,  \widetilde{h'}_{u}] $.
\end{corollary}

\begin{proof}
  The claim follows immediately from the theorem since both algebras
  are equal to the $K$-algebra $K[\col{S}{\g}]$.
\end{proof}

This independence does not hold if we substitute $\tilde{h}_{i}$ with $h_{i}$ 
for $i=1,\dots,t$, and likewise  $\widetilde{h'}_{i}$ with 
$h'_{i}$ for $i = 1,\dots, u$, as the following example shows.
Please note that in all examples using $\g=a_0$ we identify 
$K[a_0,a_1, \dots, a_n]/\ideal{a_0}$ with $K[a_1, \dots, a_n]$.

\begin{example}\label{ex:depend}
Let $P = \QQ[a_0,a_1,a_2]$, let $g_1 = a_1^2-a_0^2a_2$, $g_2 = a_1a_2-a_0$,
$g_3 = a_2^2$, $g_4 = a_1a_2^2$, $\g=a_0$,
and $S = \QQ[g_1,g_2,g_3,g_4,\g]$.
Then we have $\pig(g_1) = a_1^2$, $\pig(g_2) = a_1a_2$, $\pig(g_3)= a_2^2$, 
$\pig(g_4) = a_1a_2^2$. 

If we let $I = \Relg(g_1, \dots, g_4,\g) \subseteq
\QQ[x_1,x_2,x_3,x_4,x_5]$ 
we get 
$ I= \ideal{H_1, H_2, H_3}$ 
where $H_1 = x_2^2 -x_1x_3$, $H_2 = x_1x_3^2 -x_4^2$,
$H_3 = x_5$.
Then it is also true  that $I = \ideal{H_1,H_2',H_3}$ where 
$H_2' = H_2 +x_3H_1 = x_2^2x_3 -x_4^2$.

Let $\tilde{h}_{i}= H_i(g_1,g_2,g_3,g_4,\g)/\g$ and let 
$h_{i} = \sat{H_i(g_1,g_2,g_3,g_4,\g)}{\g}$ 
for $i = 1,2,3$. 
Similarly,
let ${\widetilde{h'}_2}^{\mathstrut}= H_2'(g_1,g_2,g_3,g_4,\g)/\g$, and
$h_2' = \sat{H_2'(g_1,g_2,g_3,g_4,\g)}{\g}$.
We have 
$$\begin{array}{ll}
\tilde{h}_{1}= H_1(g_1,g_2,g_3,g_4,\g)/\g = -2a_1a_2 +a_0(a_2^3+1) \ &
h_{1} = \sat{H_1(g_1,g_2,g_3,g_4,\g)}{\g} = \tilde{h}_{1} \\
\tilde{h}_{2}= H_2(g_1,g_2,g_3,g_4,\g)/{\g} = -a_0a_2^5&
h_{2} = \sat{H_2(g_1,g_2,g_3,g_4,\g)}{\g} \!=\!-a_2^5
\\
{\widetilde{h'}_2}^{\mathstrut}= H_2'(g_1,g_2,g_3,g_4,\g)/{\g} =-2a_1a_2^3 +a_0a_2^2 &
h_2' = \sat{H_2'(g_1,g_2,g_3,g_4,\g)}{\g} = {\widetilde{h'}_2}^{\mathstrut}
\\
\tilde{h}_{3}= H_3(g_1,g_2,g_3,g_4,\g)/\g = \g/\g = 1 \quad&
h_{3} = \sat{H_3(g_1,g_2,g_3,g_4,\g)}{\g} = 1 \\
\end{array}
$$

According to Corollary~\ref{rem:notdepend} and these equalities we have
$$\QQ[g_1,g_2,g_3,g_4, \g, h_1, h_2']  =
\QQ[g_1,g_2,g_3,g_4, \g, h_1,  \tilde{h}_2] 
\eqno{(*)}
$$
Let us check it  using Proposition~\ref{prop:subalgebraMember}.
On the polynomial ring $\QQ[x_1,\dots, x_7, a_0,a_1,a_2]$ we introduce a term 
ordering $\sigma$ of elimination for $\{a_0, a_1, a_2\}$, and we let
$$
\begin{array}{l}
J_1 = \ideal{\g,\ x_1-g_1,\, x_2-g_2, \ x_3-g_3,\ x_4-g_4,\  x_5-\g,\  x_6-h_1,\ x_7 - h_2'}\\
J_2 = \ideal{\g,\ x_1-g_1,\, x_2-g_2, \ x_3-g_3,\ x_4-g_4,\  x_5-\g,\  x_6-h_1,\ x_7 -\tilde{h}_2}
\end{array}
$$
We get $\NF_{\sigma,J_1}(\tilde{h}_2) =  -x_3x_6+x_7$ which means that 
$\tilde{h}_2 = -g_3h_1+h_2'$ and hence we deduce that 
${\tilde{h}_2}^{\mathstrut} \in \QQ[g_1,g_2,g_3,g_4,\g, h_1, h_2']$.
We get $\NF_{\sigma,J_2}(h_2') = x_3x_6+x_7$ which means that 
$h_2' = g_3h_1+{\tilde{h}_2}^{\mathstrut}$ and hence we deduce that 
$h_2' \in \QQ[a_0,g_1,g_2,g_3,g_4, h_1,  \tilde{h}_2]$.

Finally, we claim that 
$$\QQ[g_1,g_2,g_3,g_4,\g, h_1, h_2']  \subsetneq
\QQ[g_1,g_2,g_3,g_4,\g, h_1, h_2] 
$$
The inclusion $\QQ[g_1,g_2,g_3,g_4,\g, h_1, h_2'] \subseteq
\QQ[g_1,g_2,g_3,g_4,\g, h_1, h_2] 
$ follows from $(*)$  since clearly 
$\QQ[g_1,g_2,g_3,g_4,\g, h_1,  {\tilde{h}_2}^{\mathstrut}]  
\subseteq \QQ[g_1,g_2,g_3,g_4,\g, h_1, h_2] $.
Finally, to check the claim we  show that 
$h_2 \notin \QQ[g_1,g_2,g_3,g_4,\g, h_1, {h_2'}^{\mathstrut}] $. 
To do this we compute 
$\NF_{\sigma,J_1}(h_2) = -x_3^2a_2$, and the conclusion follows from
Proposition~\ref{prop:subalgebraMember}.
\end{example}

\subsection{The general Algorithm}
Theorem~\ref{thm:main} motivates the following definition.


\begin{definition}\label{def:aplus}
Given a subalgebra $S = K[g_1, \dots, g_r]$ of~$P$,
and $\g\in S{\setminus}\{0\}$,
we denote by \define{$E_\g(S)$}
the algebra $K[g_1, \dots, g_r, h_1, \dots, h_t]$ as described in 
Theorem~\ref{thm:main}.
Then we  let \define{$E_\g^0(S)$} $= S$,
and recursively  \define{$E_g^i(S)$}= \define{$E_\g(E_\g^{i-1}(S))$} for $i>0$.
\end{definition}

\begin{remark}\label{rem:abuse}
We observe that there is an abuse of notation since $E_\g(S)$
depends on the set of generators of~$S$, as shown in
Example~\ref{ex:depend}.
Moreover, we notice that the last inclusion of Theorem~\ref{thm:main}
can be read as  $E_\g(S)\subseteq \sat{S}{\g}$.
\end{remark}

We are ready to prove some fundamental results for our algorithm.

\begin{theorem}\label{thm:inclusion}
Let $P = K[a_0,a_1,\dots,a_n]$, let $g_1,\dots, g_r \in P$,
let $S = K[g_1, \dots, g_r]$, and $\g\in S{\setminus}\{0\}$.
Then let~$A$ be a finitely generated $K$-subalgebra of~$P$ such that
${S \subseteq A \subseteq \sat{S}{\g}}\!$.

\begin{enumerate}
\item We have $K[\col{S}{\g^i} ]\subseteq E_\g^i(S)\subseteq \sat{S}{\g}$ 
for every $i\ge0$.

\item We have $S = E_\g^0(S)\subseteq E_\g^1(S)\subseteq \cdots \subseteq E_\g^i(S) 
\subseteq \sat{S}{\g}$ for every $i\ge 0$.

\item We have $\sat{A}{\g} = \sat{S}{\g}$.

\item If  $A = E_\g(A)$ then  $A = \sat{S}{\g}$.
\end{enumerate}
\end{theorem}

\begin{proof} For claim (a) we have to prove two inclusions. 
For the first inclusion it suffices to show 
$\col{S}{\g^i} \subseteq E_\g^i(S)$ for $i>0$.
From Theorem~\ref{thm:main}
and Remark~\ref{rem:abuse} we  get 
$\col{S}{\g} \subseteq E_\g^1(S)$. 
By induction we may assume that 
$\col{S}{\g^{i-1}} \subseteq E_\g^{i-1}(S)$ and
let $f \in P$ be such that $\g^if \in S$. 
Then $\g f \in \col{S}{\g^{i-1}} \subseteq E_\g^{i-1}(S)$ by induction, and
hence $f \in \col{E_\g^{i-1}(S)}{\g} \subseteq E_\g^i(S)$ 
by Theorem~\ref{thm:main} applied to the subalgebra $E_\g^{i-1}(S)$.

The second inclusion of claim (a) is true for $i=0$. By induction we may
assume that $E_\g^{i-1}(S)\subseteq \sat{S}{\g}$. Let $f\in E_\g^i(S)$. 
Since $E_\g^i(S) = E_\g(E_\g^{i-1}(S))$ there exists s such that 
$\g^s f \in E_\g^{i-1}(S)\subseteq \sat{S}{\g}$. Consequently, there exists~$t$
such that $\g^{s+t} f \in S$, and hence we get  $f \in \sat{S}{\g}$.

Claim (b) follows from the definition of $E_\g^i(S)$ and claim (a).

 Claim (c) follows from Proposition~\ref{prop:MainProblem}.(c),(\ref{satequal}).

To prove claim (d) it suffices to show the inclusion $\sat{S}{\g} \subseteq A$. 
So let $ f \in \sat{S}{\g}$ which means that there exists $i \in \NN$ such that 
$\g^if\in S$. If $i =0$ we have $f \in S \subseteq A$. Using induction on $i$
we may assume that  $\g^{i-1}f \in S$ implies $f \in A$. From $\g^if \in S$
we deduce $\g^{i-1}(\g f) \in S$, hence by induction we get $\g f \in A$.
Consequently, we get $f \in \col{A}{\g} \subseteq E_\g(A)$ 
by Theorem~\ref{thm:main}.
By assumption we have $E_\g(A)= A$ hence $f \in A$ and    the proof is complete.
\end{proof}

From this theorem we deduce the following result.

\begin{corollary}\label{cor:sequence}
Let~$K$ be a field, let $P\!=\!K[a_0,a_1,\dots,a_n]$,
let $S{=}K[g_1, \dots, g_r]\subseteq P$, and let $\g\in S{\setminus}\{0\}$.
The following conditions are equivalent.
\begin{enumerate}
\item The algebra $\sat{S}{\g}$ is finitely generated.
\item There exists $i$ such that $E_\g^i(S) =E_\g^{i+1}(S)$.
\end{enumerate}
Moreover, if the two equivalent conditions are satisfied, 
then $\sat{S}{\g} = E_\g^i(S) $.
\end{corollary}

\begin{proof}
To show the implication  $(a)\!\implies\!(b)$ we assume 
that $\sat{S}{\g}= K[h_1, \dots, h_s]$ 
and let $m = \max_{i=1}^s\{i \mid \g^ih_j \in S \ {\rm for}\  j = 1, \dots,s \}$.
We deduce that 
$K[h_1, \dots, h_s] \subseteq K[\col{S}{\g^m}]$.
By Theorem~\ref{thm:inclusion}.(a) we have
$ K[\col{S}{\g^m}] \subseteq E_\g^m(S)$, hence 
$$\sat{S}{\g}= K[h_1, \dots, h_s]\subseteq K[\col{S}{\g^m}]
\subseteq E_\g^m(S)\subseteq E_\g^{m+1}(S)\subseteq \sat{S}{\g}$$
which implies the equality $E_\g^m(S) = E_\g^{m+1}(S)$.

To show that $(b)\!\implies\!(a)$ it suffices to prove that $\sat{S}{\g} = E_\g^i(S)$ 
since $E_\g^i(S)$ is finitely generated by definition. 
We have the equality $E_\g^i(S) = E_\g(E_\g^i(S))$ by assumption, 
and hence $E_\g^i(S)  = \sat{S}{\g}$ by Theorem~\ref{thm:inclusion}.(d).
\end{proof}

We are ready to describe an algorithm to compute a set of generators for
 $\sat{S}{\g}$, if it is finitely generated.
If it is not, this procedure does not terminate, producing an infinite
sequence of subalgebras ever closer to $\sat{S}{\g}$.
Instances of   $\sat{S}{\g}$ being not finitely generated are 
Example~\ref{ex:nonsatreduce} and Example~\ref{ex:notfinite}.
A similar algorithm/procedure is contained 
in~\cite[Semi-algorithm 4.10.16]{DK2015}. 
In our case we can claim that it is an algorithm since we assume that $S:\g$
is finitely generated.

\MyAlgorithm{SubalgebraSaturation}

\begin{description}[topsep=0pt,parsep=1pt] 
 \item[\textit{notation:}]  $P =K[a_0,\dots,a_n]$ is a polynomial ring.
 \item[Input]  
$S= K[g_1,\dots,g_r]\subseteq P$, 
and $\g\in S{\setminus}\{0\}$
such that $\sat{S}{\g}$ is
finitely generated.
     \item[1] Let $S' = S$
 \item[2] \textit{Main Loop:}
   \begin{description}[topsep=0pt,parsep=1pt]
   \item[2.1] call $g'_1,\dots,g'_{s}$ the non-constant generators of $S'$
   \item[2.2] compute $\{H_1, \dots, H_{t}\}$,
     a set of generators of the ideal
     $\Relg(g'_1,\dots,g'_{s})$
   \item[2.3] for $j=1,\dots,t$, let $h_j = \sat{H_j(g'_1, \dots, g'_{s})}{\g}$

   \item[2.4] if $h_1,\dots,h_{t}\in S'$, i.e.\ if $E_\g(S') = S'$, 
          then \define{return $S'$}
   \item[2.5] redefine $S'$ as $K[g'_1,\dots,g'_{s}, h_{1},\dots,h_{t}]$
   \end{description}
 \item[Output ] $\sat{S}{\g}$
\end{description}

\begin{proof}
Since $\sat{S}{\g}$ is finitely generated,  correctness and
termination follow immediately from Corollary~\ref{cor:sequence}.
\end{proof}

\begin{remark}\label{rem:gindet}
When $\g$ is indeed in the list~$G$ of the generators of the
subalgebra~$S$, say in position $i$, then $x_i$ is in $\Relg(G)$.
Then $h_i = 1$ which trivially belongs to~$S$.
In the following examples we will use this fact sistematically.
\end{remark}

The following example shows that the procedure may not terminate, 
and $\sat{S}{a_0}$ 
needs not be a finitely generated $K$-algebra.

\begin{example}\label{ex:nonsatreduce}
Let $P = \QQ[a_0,a_1,a_2]$, let 
$\g=a_0$, 
$g_2 = a_1-a_0a_1^2$, 
$g_3 = a_2$, 
$g_4 = a_1a_2$, and let $S = \QQ[\g, g_2, g_3, g_4]$.
Notice that $P/\ideal{\g} \simeq \QQ[a_1,a_2]$.
Then~we have $\pig(g_2) = a_1$, $\pig(g_3) = a_2$, $\pig(g_4)= a_1a_2$, 
hence  $\Relg(\g, g_2, g_3, g_4) = \ideal{x_1,\, x_4-x_2x_3}$.
First, from $H_1 = x_1$ we have $h_1 =1\in S$ (as shown in Remark~\ref{rem:gindet}).
Then, from $H_2= x_4-x_2x_3$ 
we have $g_4-g_2g_3 = a_0(a_1^2a_2)$ hence $h_2 = a_1^2a_2$.
 Therefore, after the first loop, 
$E_\g(S) = \QQ[a_0,\,a_1{-}a_0a_1^2,\, a_2,\, a_1a_2,\, {a_1^2a_2}^{\mathstrut}]$.
Inductively, we may assume that 
$$E_\g^i(S) = \QQ[a_0,\,
a_1{-}a_0a_1^2, \, a_2, \,a_1a_2, \,a_1^2a_2,\dots, {a_1^{i+1}a_2}^{\mathstrut}]$$
The only new relation in $\Relg(\g,g_2,g_3,g_4, a_1^2a_2,\dots, a_1^i a_2, a_1^{i+1}a_2)$
is $x_2x_{i+3} -x_{i+4}$ and after the loop we get $a_1^{i+2}a_2$.
The procedure does not stop, nevertheless we can conclude that 
$$
\sat{S}{a_0} = 
\QQ[a_0,\,a_1{-}a_0a_1^2,\, a_2,\, a_1a_2,\, a_1^2a_2,\dots, a_1^{i+1}a_2,\dots]
$$
hence it is not finitely generated.
\end{example}

Let us see an example where the procedure stops, hence it computes $\sat{S}{a_0}$.
\begin{example}\label{ex:terminates}
Let~$P =\QQ[a_0,a_1,a_2,a_3]$, let $\g=a_0$, $g_2 = a_1^2-a_0a_2$,
$g_3 =a_1^3-a_0a_3$, and let $S =\QQ[\g,g_2,g_3]$. 
We have $\pig(g_2) = a_1^2$
and $\pig(g_3) = a_1^3$, hence 
$\Relg(\g, g_2, g_3) = \ideal{x_1,\,x_3^2 -x_2^3}$.
We get 
$g_3^2-g_2^3 = 3a_0a_1^4a_2 -2a_0a_1^3a_3 -3a_0^2a_1^2a_2^2
+a_0^2a_3^2 +a_0^3a_2^3$,
hence 
$g_4 = a_1^4 a_2 -{\frac{2}{3} }^{\mathstrut} a_1^3 a_3 -a_0  a_1^2
a_2^2  +\frac{1}{3}  a_0  a_3^2  +\frac{1}{3}  a_0^2 a_2^3$, 
and hence we deduce $E_\g(S) = K[\g, g_2,g_3,g_4]$,
and indeed we can check that $g_4\not\in S$.
Moreover, we have $\pig(g_4) = a_1^4a_2 -{\frac{2}{3}}^{\mathstrut} a_1^3a_3$ and 
$\Relg(\g, g_2, g_3, g_4)=\ideal{x_2^2 -x_1^3}$, so no new generator 
is created in Step 2.2 and the procedure stops in Step 2.3. 
The conclusion is that ${\sat{S}{a_0} = K[\g, g_2, g_3, g_4]}$.

Moreover, from the computation we deduce the equality
$a_0g_4 = \frac{1}{3}(g_3^2-g_2^3)$ 
which gives an explicit proof of the fact that 
$g_4 \in \sat{S}{a_0}$.
\end{example}

Algorithm~\ref{alg:SubalgebraSaturation} comes as a direct application 
of the theory developed in Section~\ref{sec:The General Case}, 
in particular Corollary~\ref{cor:sequence}.  It is useful to improve it by using 
suitable rewriting procedures which we are going to describe in the next section.

\section{Subalgebra Reduction, Interreduction, and Sat-Interreduction}
\label{sec:Subalgebra-reduction}

We recall some definitions and facts from the theory of \sagbis.
For a general introduction to this topic see~\cite[Section 6.6]{KR2};
here, we reshape its Definition~6.6.16 and adapt it
for our purposes.

\begin{definition}\label{def:reduction-step}
Let $P =K[a_0,\dots,a_n]$ 
with term-ordering~$\sigma$.
Let $G= \{g_1, \dots, g_r\}$, where all $g_i$'s are monic 
polynomials in~$P$,
and let~$h$ be a non-zero polynomial in~$P$.
If $\LT_\sigma(h)\in K[\LT_\sigma(g_1),\dots, \LT_\sigma(g_r)]$,
and we have 
$\LT_\sigma(h) = \LT_\sigma(g_1)^{\alpha_1}\cdots
\LT_\sigma(g_r)^{\alpha_r}$,
then
we let $h' =h - \LC_\sigma(h)\cdot g_1^{\alpha_1}\cdots g_r^{\alpha_r}$ and
we say that the passage from~$h$ to~$h'$ is  an
\define{$\mathcal{S}_{\LT}$-reduction step} for~$h$.
\end{definition}

The following is a running example for this section.

\begin{example}\label{ex:sagbirem-step}
Let $P =\QQ[a_0,a_1,{a_2}_{\mathstrut}]$, with $\TT(a_0,a_1,a_2)$ ordered by
$\sigma$, the term-ordering defined by  the matrix
 $\left(
\begin{smallmatrix}
\  \ 1&\ \  1&\ \  1 \\
 -1&\ \  0& \ \ 0 \\ 
\ \ 0&-1&\ \  0
\end{smallmatrix}
\right)
$.
 Then, let 
$h= a_1a_2^6 -4a_0^5a_1a_2 +4a_0^5a_1^2 +a_0^6a_2 +a_0^7$,
and
$$
g_1=a_0, \quad
g_2=a_1a_2 -a_1^2, \quad
g_3= a_2^2, \quad
g_4= a_1a_2^2
$$
We observe that all polynomials are monic and we have 
$$
\LTs(g_1)= a_0, \quad 
\LTs(g_2)= a_1a_2, \quad 
\LTs(g_3)= a_2^2,  \quad
\LTs(g_4)= a_1a_2^2, \quad
\LTs(h)= a_1a_2^6,
$$
We observe that
  $\LTs(h) = \LTs(g_3)^2 \LTs(g_4)$.
Hence  we have an $\mathcal{S}_{\LT}$-reduction step 
$h' =h - {g_3^2}^{\mathstrut} g_4 = -4a_0^5a_1a_2 +4a_0^5a_1^2 +a_0^6a_2 +a_0^7$.
\end{example}

Note that an $\mathcal{S}_{\LT}$-reduction step replaces $\LT_\sigma(h)$ with
$\sigma$-smaller terms.
Therefore, being $\sigma$ a term ordering, a chain of LT-reduction
steps must end in a finite number of steps.
This motivates the following definitions.

\begin{definition}\label{def:Subalgebra-reduction}
Let $P =K[y_1,\dots,y_m]$, and let $\sigma$  be a term ordering on
$\TT(y_1,\dots, y_m)$.
Then let $G= \{g_1, \dots, g_r\}$, where all $g_i$'s are monic 
polynomials in~$P$ and let $0\ne h \in P$.
\begin{enumerate}
\item 
We say that $h'$ is an \define{\SLTrem} for~$h$ and 
 denote it by $\saremLT(h,G)$, if there is 
a chain of $\mathcal{S}_{\LT}$-reduction steps from~$h$ to $h'$,
and $\LT_\sigma(h') \not\in K[\LT_\sigma(g_1),\dots, \LT_\sigma(g_r)]$.

\item
Let $h' =\saremLT(h,G)$.  We may compute 
$\LM_\sigma(h')-\saremLT(h'-\LM_\sigma(h'), G)$,
and repeat this process until we obtain a polynomial $h''$ such that no
power-product in its support is in $K[\LT_\sigma(g_1),\dots, \LT_\sigma(g_r)]$.
We say that $h''$ is an \define{\Srem} for~$h$
 and denote it by $\sarem(h,G)$.
\end{enumerate}
\end{definition}

\begin{remark}\label{rem:different remainders}
Notice that, according to  Definition~\ref{def:reduction-step}, we may get
different $\saremLT(h,G)$ and $\sarem(h,G)$, depending
on the way of representing 
$\LT_\sigma(h) = 
\LT_\sigma(g_1)^{\alpha_1}\cdots \LT_\sigma(g_r)^{\alpha_r}$. 
\end{remark}

This definition is a natural generalization of the remainder of the
division algorithm in the context of polynomial ideal, 
but the difficult step here is to find 
the $\alpha_i$ giving the equality 
$\LT_\sigma(h) = 
\LT_\sigma(g_1)^{\alpha_1}\cdots \LT_\sigma(g_r)^{\alpha_r}$.

There are two strategies for doing this: 
elimination and toric ideals.

\begin{remark}
As described in Proposition~\ref{prop:subalgebraMember}
we may determine an explicit 
representation of $\LT_\sigma(h)$ as an element
of~$K[\LT_\sigma(g_1), \dots, \LT_\sigma(g_r)]$ by
defining the ideal
 $J = \ideal{x_1-\LT_\sigma(g_1), \dots, x_r-\LT_\sigma(g_r)}$
in $K[x_1, \dots, x_r, a_0, \dots, a_n]$
 and then computing
$\tau=\NF_{\sigma, J}(\LT(h))$.
It is easy to show that $\tau$ is a power-product, and,
if $\tau = x_1^{\alpha_1}\cdots x_r^{\alpha_r}$
we have the desired exponents.
 Otherwise, if some $a_i$ occurs in
$\tau$, we may conclude that $\LT_\sigma(h) \not\in K[\LT_\sigma(g_1), \dots, \LT_\sigma(g_r)]$,
\ie~there is no $\mathcal{S}_{\LT}$-reduction step for $h$.
\end{remark}

The other strategy is by looking for a binomial $x_0-\tau$ in the
kernel of the map 
$\phi: K[x_0,x_1,\dots,x_r] \longrightarrow P$, defined by $\phi(x_0)
= \LT_\sigma(h)$ and $\phi(x_i) = \LT_\sigma(g_i)$.

Again this may be computed by mimicking
Proposition~\ref{prop:subalgebraRepr}, but much more efficiently by
computing the ${\rm toric\  ideal\  } \Rel(\LT_\sigma(h), \,\LT_\sigma(g_1),\,
\dots,\,\LT_\sigma(g_r))$ as described for instance 
in~\cite{BLR}, and using the following proposition.

\begin{proposition}\label{prop:toric}
Let $t_0,t_1,\dots, t_r$ be power-products in $P$
and $\phi\!:\! K[x_0,x_1,\!\dots,x_r] \longrightarrow P$ be the $K$-algebra 
homomorphism defined by $\phi(x_i)=t_i$ for $i=0,\dots, r$.
Then the following conditions are equivalent.
\begin{enumerate}
\item We have $t_0 \in K[t_1, \dots, t_r]$.

\item
 There exists a binomial $b\in \ker(\phi)$
such that $x_0 \in \Supp(b)$.

\item Given a finite set $B$ of binomial generators of $\ker(\phi)$,
there exists a binomial $b\in B$ such that $x_0 \in \Supp(b)$.
In this case, if
$b= \pm(x_0 - \tau)$, then $t_0 = \tau(t_1,\dots, t_r)$.
\end{enumerate}
\end{proposition}

\begin{proof}
The implications $(c)\implies (b)$ and $(b)\implies (a)$ are clear. 

Let us prove $(a)\implies (c)$.
The multi-homogeneity of $K[t_1, \dots, t_r]$ implies that the only way to have 
$t_0 \in K[t_1, \dots, t_r]$ is to have an equality of 
type $t_0=\prod_{i=1}^rt_i^{\alpha_i}$. 
This implies~(b), \ie~$x_0-\prod_{i=1}^rx_i^{\alpha_i} \in \ker(\phi)$.
Given $B = \{b_1, \dots, b_t\}$, we get 
$x_0-\prod_{i=1}^rx_i^{\alpha_i} = \sum_{i=1}^t f_ib_i$ 
with $f_i \in K[x_0,\dots, x_r]$.
By putting $x_1=\cdots =x_r=0$ in this relation  
we see that one of the $b_i$ has to be either of type 
$\pm(x_0-\prod_{i=1}^rx_i^{\beta_i})$, 
or of type $\pm(1-\prod_{i=1}^rx_i^{\beta_i})$.
The latter is excluded by the homogeneity 
of the generators of $\ker(\phi)$, therefore the proof is complete.
\end{proof}

\begin{remark}\label{rem:only submultiples}
There is an obvious but practically effective improvement of
Proposition~\ref{prop:toric}.
Asking whether $t_0 \in K[t_1, \dots, t_r]$
is equivalent to asking whether $t_0 \in K[T]$
where $T = \{t_i \mid t_i \text{ divides }t_0\}$.
\end{remark}

\begin{example}\label{ex:sagbirem}(Example~\ref{ex:sagbirem-step}, continued)\\
Let $P =\QQ[a_0,a_1,a_2]$, $\sigma$, $g_i$ and~$h$ as in
Example~\ref{ex:sagbirem-step}.
After the first $\cal{S}_{\LT}$-reduction step we got
$h' =h - g_3^2 g_4 = -4a_0^5a_1a_2 +4a_0^5a_1^2 +a_0^6a_2 +a_0^7$.

Now, $\LTs(h') = a_0^5a_1a_2 = \LTs(g_1)^5\LTs(g_2)$ gives us
a following $\cal{S}_{\LT}$-reduction step:
$h'' = h' +4 g_1^5g_2 = a_0^6a_2 +a_0^7$, whose leading term cannot
be further reduced.  Therefore $\saremLT(h, G) = h'' = a_0^6a_2 +a_0^7$.

Setting apart its leading monomial, we can now consider 
$\tilde h = h'' - \LM_\sigma(h'')$, 
and we see that  $\tilde h = \LTs(g_1)^7 = g_1^7$, therefore  
$\saremLT(\tilde h, G)= 0$.
In conclusion, $\sarem(h, G) = a_0^6a_2$.
\end{example}

\subsection{Interreduction and Sat-Interreduction}
\label{sec:Sat-Interreduction}

We recall from Definition~\ref{def:aplus} that $E_\g(A)$ is obtained by adding 
new generators to those of the algebra~$A$.
Now we investigate on how, using $\mathcal{S}$-remainders, we can
find a new set of generators for~$E_\g(A)$ or, even better, a 
set of polynomials generating an algebra~$B$ such that 
$E_\g(A) \subseteq B \subseteq \sat{A}{a_0}$

\begin{proposition}\label{prop:interreduction}
Let $P \!=K[a_0,a_1,\dots,a_n]$, with term ordering $\sigma$ on $\TT^{n+1}$,
and let $S\!=\! K[g_1, \dots, g_r]$, with $g_i\in P$, 
and $\g\in S{\setminus}\{0\}$.
Let $g'_{i} =  
\sarem(g_{i}, \{g_1, \dots, \hat{g_i}, \dots, g_r\})$,
 then 
$$S \;\subseteq \;K[g_1, \dots, \sat{g'_i}{\g}, \dots g_r]\; \subseteq\; \sat{S}{\g}$$
\end{proposition}

\begin{proof}
From the definition of \Srem it is clear that  $S= K[g_1, \dots, g_i', \dots g_r]$,
Then the assumption $\g\in S$ implies the inclusion
$K[g_1, \dots, g_i', \dots g_r]\subseteq K[g_1, \dots, \sat{g'_i}{\g}, \dots g_r]$,
and the conclusion follows.
\end{proof}

\begin{definition}\label{def:gradedaplus}
Let $P \!=K[a_0,a_1,\dots,a_n]$, with term ordering $\sigma$,
and $S= K[g_1, \dots, g_r]$, where all $g_i$'s are monic 
polynomials in~$P$.
Repeating the substitution described in
Proposition~\ref{prop:interreduction}
 until no more $\mathcal{S}$-reductions 
and saturations are possible, we obtain a set of  
sat-$\mathcal{S}$-interreduced generators of a 
$K$-algebra~$A$ such that 
$S \subseteq A \subseteq \sat{S}{a_0}$.

We denote such $A$ by \define{$\satinterr(S)$}, and
again, as in Definition~\ref{def:aplus}, there is an abuse of
notation since $\satinterr(S)$ 
depends on the set of generators of~$S$ and also on the steps of reduction.
\end{definition}

The following easy example illustrates this definition.

\begin{example}\label{ex:nonhomog}
Let $P =K[a_0,a_1,a_2]$ with~$\tt{DegRevLex}$.
Let 
$S= K[a_0,\; a_1,\;a_0a_2^2 {-}a_1]$,
then  $\satinterr(S) =K[a_0, \;a_1, \;a_2^2]=\sat{S}{a_0}$.
\end{example}

In general, 
the set of interreduced generators is obtained after 
more than one iteration through the generators.

\begin{example}\label{ex:sagbiremDegRev}(Examples~\ref{ex:sagbirem-step} and~\ref{ex:sagbirem} continued)\\
As in Example~\ref{ex:sagbirem-step}.
we let $P =\QQ[a_0,a_1,{a_2}_{\mathstrut}]$, with $\TT(a_0,a_1,a_2)$ ordered by
$\sigma$, the term-ordering defined by  the matrix
 $\left(
\begin{smallmatrix}
\  \ 1&\ \  1&\ \  1 \\
 -1&\ \  0& \ \ 0 \\ 
\ \ 0&-1&\ \  0
\end{smallmatrix}
\right)$. 
Let 
$g_1=a_0, \;
g_2=a_1a_2 -a_1^2, \;
g_3= a_2^2, \; \allowbreak
g_4= a_1a_2^2
$,
let $g_5=h= a_1a_2^6 -4a_0^5a_1a_2 +4a_0^5a_1^2 +a_0^6a_2 +a_0^7$,
and let $A = K[g_1,g_2,g_3,g_4,g_5]$.

Notice that 
$\LT(g_1)< \LT(g_2)< \LT(g_3)< \LT(g_4)< \LT(g_5)$,
thus $g_i$ may be reduced only by the $g_j$'s with $j<i$.
The only one which may be reduced is $g_5$,
and we computed $\sarem(h, G) = a_0^6a_2$
(see Example~\ref{ex:sagbirem}),
whose saturation is~$a_2$.

Then, we re-sort and re-number the $g_i$'s:
$$
g_1 = a_0,\quad
g_2 = a_2, \quad
g_3 = a_1a_2 -a_1^2, \quad
g_4 = a_2^2,\quad
g_5 = a_1a_2^2\ .
$$ 
Now, we see that $g_4$ can be $\cal{S}$-reduced to 0 using $g_2$,
and $g_5$ can be $\cal{S}$-reduced to $a_1^2a_2$ using $g_2g_3$.
Re-sorting and re-numbering again, we have
$$
g_1 = a_0,\quad
g_2 = a_2, \quad
g_3 = a_1a_2 -a_1^2, \quad
g_4 = a_1^2a_2 \ .
$$
In conclusion we have 
$
\satinterr(A) = k[a_0,\; a_2,\; a_1a_2-a_1^2,\; a_1^2a_2]
$.
Moreover, it is easy to see that $\satinterr(A)  = \sat{A}{a_0}$.
\end{example}

As said, the process of sat-interreducing the generators of a subalgebra of~$P$
can improve the subsequent steps of the computation of $\Relg$. But we cannot hope
that it substitutes such computation, as the following example  shows.

\begin{example}\label{ex:satinterrednotenough}
Let $P =\QQ[a_0,a_1,{a_2}_{\mathstrut}]$, with $\TT(a_0,a_1,a_2)$ ordered by
$\sigma$, the term-ordering defined by  the matrix
 $\left(
\begin{smallmatrix}
\begin{smallmatrix}
\  \ 1&\ \  1&\ \  1 \\
 -1&\ \  0& \ \ 0 \\ 
\ \ 0&\ \  0&  -1
\end{smallmatrix}
\end{smallmatrix}
\right)$. 
Let $G = \{\g, g_2, g_3, g_4\}$ where
$
\g=  a_0,                     \
g_2=  a_1a_2 -a_0a_1 +a_0a_2, \
g_3=  a_1^2 -a_2^2 +a_0a_1,    \
g_4=  a_1^3 -a_0a_2^2          
$,
and $S = K[G]$.

Then $\Relg(G)=\ideal{
x_1,\, x_2^6 +3x_2^2x_3x_4^2 +x_3^3x_4^2 -x_4^4
}$

The second relation, evaluated in~$G$, and $a_0$-saturated gives $h$,
and   $\sat{\sarem(h,G)}{a_0} = \sarem(h,G) = $
{\tiny
$
a_1^5 a_2^6  -6  a_1^4 a_2^7  -a_1^3 a_2^8  +40  a_0  a_1^4 a_2^6  -2  a_0  a_1^3 a_2^7  -45  a_0  a_1^2 a_2^8  -40  a_0  a_1  a_2^9  -10  a_0  a_2^{10}  +107  a_0^2 a_1^5 a_2^4  +484  a_0^2 a_1^4 a_2^5  +441  a_0^2 a_1^3 a_2^6  +38  a_0^2 a_1^2 a_2^7  -101  a_0^2 a_1  a_2^8  -40  a_0^2 a_2^9  -808  a_0^3 a_1^3 a_2^5  +615  a_0^3 a_1^2 a_2^6  +116  a_0^3 a_1  a_2^7  -39  a_0^3 a_2^8  -3798  a_0^4 a_1^4 a_2^3  +4935  a_0^4 a_1^3 a_2^4  -3846  a_0^4 a_1^2 a_2^5  +304  a_0^4 a_1  a_2^6  +82  a_0^4 a_2^7  +23372  a_0^5 a_1^2 a_2^4  -2720  a_0^5 a_1  a_2^5  +90  a_0^5 a_2^6  -50860  a_0^6 a_1^3 a_2^2  +5256  a_0^6 a_1^2 a_2^3  +30105  a_0^6 a_1  a_2^4  -166  a_0^6 a_2^5  +78690  a_0^7 a_1  a_2^3  +8438  a_0^7 a_2^4  -228828  a_0^8 a_1^2 a_2 +63304  a_0^8 a_1  a_2^2  +77232  a_0^8 a_2^3  +258692  a_0^9 a_2^2  -369708  a_0^{10} a_1 +228828  a_0^{10} a_2
$
}.

But $\sarem(h,G)$ is indeed in~$S$, being $\sarem(h,G) = $
{\tiny
$
228828 \g^9 g_{2} +85356 \g^7 g_{2}^2  +14530 \g^5 g_{2}^3  +1492 \g^3
g_{2}^4  -72 \g g_{2}^5  -140880 \g^9 g_3 -58116 \g^7 g_{2} g_{3}
-8051 \g^5 g_{2}^2 g_{3} +38 \g^3 g_{2}^3 g_{3} +151 \g g_{2}^4 g_{3}
-2592 \g^7 g_{3}^2  -1576 \g^5 g_{2} g_{3}^2  +1009 \g^3 g_{2}^2
g_{3}^2  -180 \g g_{2}^3 g_{3}^2  +529 \g^5 g_{3}^3  -492 \g^3 g_{2}
g_{3}^3  +99 \g g_{2}^2 g_{3}^3  +114 \g^3 g_{3}^4  -46 \g g_{2}
g_{3}^4  +11 \g g_{3}^5  -32456 \g^8 g_{4} -4586 \g^6 g_{2} g_{4}
-3263 \g^4 g_{2}^2 g_{4} +740 \g^2 g_{2}^3 g_{4} -1751 \g^6 g_{3}
g_{4} +2196 \g^4 g_{2} g_{3} g_{4} -654 \g^2 g_{2}^2 g_{3} g_{4} +18
g_{2}^3 g_{3} g_{4} -424 \g^4 g_{3}^2 g_{4} +232 \g^2 g_{2} g_{3}^2
g_{4} -3 g_{2}^2 g_{3}^2 g_{4} -50 \g^2 g_{3}^3 g_{4} +6 g_{2} g_{3}^3
g_{4} -g_{3}^4 g_{4} -25 \g^5 g_{4}^2  +64 \g^3 g_{2} g_{4}^2  -22 \g
g_{2}^2 g_{4}^2  -27 \g^3 g_{3} g_{4}^2  +22 \g g_{2} g_{3} g_{4}^2
-8 \g g_{3}^2 g_{4}^2  +14 \g^2 g_{4}^3  -6 g_{2} g_{4}^3  +g_{3} g_{4}^3 
$
}.

\end{example}

\section{The Graded Case: Introduction}
\label{sec:The Graded Case: Introduction}

As mentioned in the introduction, the problem of computing the 
saturation of a subalgebra $S$ of $P$ can benefit from the fact that 
$S$ is graded. In this section we prepare the ground for new results
related to our problem.

We introduce here the language of (multi or single) positive gradings.
All our results in this and the following section are valid for every positive grading and, in
particular, 
for every (single) grading defined by a
row-matrix of positive weights.

We recall that a grading on~$P$ defined by a weight matrix $W$
is called \define{positive} if no 
column of $W$ is zero and the first (from the top) non-zero element in each 
column is positive.
In this case, we shall also say that W is a positive matrix.
For more on positive gradings see~\cite[Chapter 4]{KR2}.
In this section we assume that~$P$ has a positive grading
such that our algebra~$S$ is generated by homogeneous elements.

From now on we consider $P = K[a_0,a_1, \dots, a_n]$
with a positive (multi)grading defined by~$W$,
and $S=K[g_1,\dots, g_r]$ a $W$-graded $K$-subalgebra of $P$,
generated by $W$-homogeneous $g_i$'s.
It is easy to prove that there exist positive row-matrices $W'$,
so that every $W$-homogeneous polynomial is also $W'$-homogeneous.
Therefore, alongside $W$ we can consider a (single)grading 
defined by such a $W'$
which gives positive integer degree to every non-constant polynomial in~$P$.
The following easy example illustrates this claim.

\begin{example}\label{ex:positive-grading}
Let $P=\QQ[a_0,a_1]$ be graded by 
$W = \left(\begin{smallmatrix} \ \ 1 &\ 0 \\  -7 &\    2 \end{smallmatrix}\right)$,
and let $W'=(1\ 2)$.  Notice that $W' = 8(1\ 0) + (-7\ 2)$,
thus every $W$-homogeneous element in $P$ is also $W'$-homogeneous.
\end{example}

Suppose that we have a positive grading on $P$ and that $S$ 
is a $K$-subalgebra of $P$ generated by homogeneous elements.
Are there advantages depending on this assumption?

\begin{remark}\label{rem:HomogeneousRels}
Let $S = K[g_1,\dots,g_r]$ subalgebra of $P$,
with $g_i$, non constant and homogeneous of degree $d_i$.
Consider on the ring $R = K[x_1,\dots,x_r]$ the grading defined by 
the matrix $W = (d_1 \;
\dots \; d_r)$.
Then, for every term ordering $\sigma$ on $P$, the toric  ideal
$\Rel(\LT_\sigma(g_1),\dots,\LT_\sigma(g_{s})) \subseteq R$ is
homogeneous,
and any homogeneous relation of degree $d$ in $R$ evaluated in 
$(g_1,\dots,g_r)$,
gives a homogeneous polynomial of degree $d$ in~$P$.
\end{remark}

In Example~\ref{ex:nonsatreduce} the polynomial $g_1=a_1-a_0a_1^2$
is not homogeneous with respect to any positive grading, and 
no term ordering~$\sigma$ is such that $\LTs(\pig(g_1)) = a_1$. 
Example~\ref{ex:nonsatreduce} was used to show that 
$\sat{S}{a_0}$ needs not to be finitely generated.
However, the following example shows that $\sat{S}{a_0}$ 
needs not be a finitely generated $K$-algebra even 
when~$S$ has a positive grading.
It is inspired by the similar Example 6.6.7 contained in~\cite[Section 6]{KR2}.

\begin{example}\label{ex:notfinite}
Let~$P=\QQ[a_0,a_1,a_2]$ be graded by 
$W = \left(\begin{smallmatrix} 
0& 1& 1\\ 
1& 1& 0  \end{smallmatrix} \right)$.

Let
$G = \{\g, g_2, g_3, g_4\}\subseteq P$ and $S = \QQ[G]$ where
$$\g=a_0,\quad g_2=a_1 {+} a_0a_2, \quad g_3= a_1a_2, \quad g_4=a_1a_2^2$$

From
 $\pig(g_2)=a_1$,\; $\pig(g_3)=g_3$,\; $\pig(g_4)=g_4$,\;
it follows that  $E_\g^0(S) = S$ and  
${\Relg}(G) =\Rel(\g,\; a_1,\; a_1a_2,\; a_1a_2^2)
=\ideal{x_1, \;x_2x_4-x_3^2}$.
The first generator gives $\g$, and the second gives
 $g_2g_4-g_3^2= a_0a_1a_2^3$ whose saturation is~$a_1a_2^3$.
No sat-reduction 
is possible
hence we obtain 
$G_1 = G \cup \{a_1a_2^3\}$,
and $E_\g^1(S)=\QQ[G_1]$.

By induction on $i$ we assume that 
$G_{i}= \{a_0, a_1, a_1a_2, \dots,  a_1a_2^{i+2}\}$, and
$E_\g^i(S)= \QQ[G_i]$.
We prove that 
$E_\g^{i+1}(S)= \QQ[a_0, a_1, a_1a_2, \dots,  a_1a_2^{i+3}]$.

Induced by $W$, we have a grading on $\QQ[x_1,x_2,\dots, x_{i+4}]$ given by 
$V = \left(  \begin{smallmatrix}
 0 & 1 & 2 & \cdots & i+3 \\
 1 & 1 & 1 & \cdots & 1 \end{smallmatrix} \right)$
so that  
${\Relg}(G_i)$ is $V$-homogenous, 
and we consider the term ordering $\sigma$ 
defined by a matrix whose first two lines are the lines of $V$, and
the third line is $(-1\  \ 0\  \cdots \ 0)$.
Then $\sigma$ is $\deg_V$-compatible.


It is well-known that toric ideals are generated by pure 
binomials (see~\cite{Stu96}).
We claim that the binomials in the 
reduced $\sigma$-Gr\"obner basis  $\mathcal{G}_\sigma$ of the ideal ${\Relg}(G_i)$ are 
quadratic, i.e.\ of type  $x_\alpha x_\beta - x_\gamma x_\delta$. 

To prove this claim we assume by contradiction that there 
is a pure binomial $b$ in $\mathcal{G}_\sigma$  of type 
$b=x_{\alpha_1}x_{\alpha_2}\cdots x_{\alpha_r} - x_{\beta_1}x_{\beta_2}\cdots x_{\beta_s}$
with $\alpha_1\le \alpha_2\cdots \le \alpha_r$ and $\beta_1\le \beta_2 \le \cdots \le \beta_s$ and $r>2$.
From the homogeneity with respect to the second line of~$V$ we deduce 
the equality $r=s$. 
As $b$ is $\sigma$-monic, 
$\LTs(b)= x_{\alpha_1}x_{\alpha_2}\cdots x_{\alpha_r}$,
hence
 $\alpha_1<\beta_1$.
Moreover, from the homogeneity
with respect to the 
first row of $V$
 we deduce 
 $\sum_{i=1}^r\alpha_i =\sum_{i=1}^r\beta_i $ hence
we get the inequality $\alpha_r\ge\alpha_1+2$. 
Now, $H=x_{\alpha_1}x_{\alpha_r} - x_{\alpha_1+1}x_{\alpha_r-1}\in\Relg(G_i)$,  
and, since $\LTs(H)=x_{\alpha_1}x_{\alpha_r}$ divides~$\LTs(b)$ properly, 
we have a contradiction with the assumption that  $b \in \mathcal{G}_\sigma$.
Therefore,
all  binomials in $\mathcal{G}_\sigma$ are quadratic.
 
 The next claim is that only  $a_1a_2^{i+3}$ is added to 
 $E_\g^{i}(S)$ 
via the  binomials in 
$\Relg(G_i)$.
The only quadratic binomials which produce non-zero polynomials
are those involving $x_2$, and those not in $\Relg(G_{i-1})$ 
must involve $x_{i+4}$.
Thus, using the same arguments as above, we  deduce that they are 
of type $x_2x_{i+4}-x_ax_b$ with $1<a\le b<i+4$. 
The corresponding evaluation gives $a_0a_ia_2^{i+3}$
which sat-reduces to $a_1a_2^{i+3}$, and hence we have
$E_\g^{i+1}(S)= \QQ[a_0,a_1+a_0a_2, a_1a_2, \dots, a_1a_2^{i+3}]$
as claimed. 

In conclusion, $\sat{S}{a_0}
= \QQ[a_0, \, a_1 {+} a_0a_2, \, a_1a_2^2,\dots, a_1a_2^i,\dots ]$
is not finitely generated.
\end{example}

The following is a well-known fact  which we recall here for the 
sake of completeness. It states that the degrees of a 
minimal system of homogeneous generators of  a graded $K$-subalgebra
$S$ of $P$  
is an invariant of $S$. For simplicity we state it here only in the special case where 
the grading is given by a positive row-matrix.

\begin{proposition}\label{prop:mingens-same-degrees}
Let $W$ be a positive row-matrix and let $S$ be a $W$-graded finitely 
generated $K$-subalgebra of $P$.
Then let $(g_1, \dots, g_r)$ be  a minimal system of homogeneous  
generators of~$S$ with $d_i=\deg_W(g_i)$ and  $d_1\le  \cdots \le d_r$. 
If $(h_1,\dots, h_s)$ is another minimal  system of homogeneous generators
of $S$ with $\delta_i = \deg(h_i)$ and $\delta_1\le \cdots \le \delta_s$, 
then $s = r$ and $d_i = \delta_i$ for $i = 1,\dots, r$.
\end{proposition}

\begin{proof}
It suffices to show that in each degree $d$ the number of 
elements of degree $d$ in any minimal system of generators 
of $S$ is an invariant. 
Let $K[S_{< d}]$ be the algebra generated by the elements of $S$ of degree less than $d$,
and let $V = S_d \cap K[S_{< d}]$.
It is a $K$-vector subspace of $S_d$ and
the number of minimal generators of degree $d$ 
is $\dim_K(S_d/V)$.
\end{proof}

\section{The Graded Case: Truncated \sagbi for Minimalization}
\label{sec:Minimalization}

Given homogeneous generators of a $K$-subalgebra $S$ of $P$, the next question is how to
find a minimal system of homogeneous generators of $S$.
The best tool for tackling this problem
is a truncated \sagbi of $S$. Let us see how.
As mentioned,
for a general introduction to this topic see~\cite[Section 6.6]{KR2}.  
In particular, 
consider~\cite[Tutorial 96]{KR2}.

\begin{remark}\label{rem:IncreasingDegreeSAGBI}
Recall Remark~\ref{rem:HomogeneousRels}.
Starting with homogeneous generators,
the computation of a \sagbi may proceed by increasing degrees:
after all relations and generators of degree~$\le d$
have been considered, the computation continues with relations and
polynomials of higher degrees.
Thus, the following generators and relations, have
degree~$>d$, and cannot affect, \ie~reduce, those,
previously considered, of degree~$\le d$.

One application of this approach is that 
one can determine whether an element of degree~$d$
is in $S$ by testing if it reduces to $0$ or not
with respect to a \define{$d$-truncated \sagbi} of $S$,
\ie~a \sagbi computed up to degree $d$.
\end{remark}

With these facts
we are ready to describe the algorithm for computing the
minimal generators.
This algorithm is basically the same as the general algorithm for computing
a \sagbi, except for the considerations on the degree.

\MyAlgorithm{SubalgebraMinGens}
 
\begin{description}[topsep=0pt,parsep=1pt] 
 \item[\textit{notation:}]  $P =K[a_0,\dots,a_n]$ is a polynomial
   ring graded by a positive row-matrix,\\
and let~$\sigma$ on $\TT^{n+1}$ be a degree-compatible term ordering.
 \item[Input]  
$S = K[g_1,\dots,g_r]\subseteq P$
with $g_1,\dots, g_r$ homogeneous.

\item[1] Initialise: 
Let $G = \{g_1,\dots,g_r\}$,
$D  = \max\{\deg(g)\mid g\in G\}$,
 SB = $\emptyset$, and MinGens = $\emptyset$.
 \item[2] \textit{Main Loop:} for $d$ 
   \;from $\min\{\deg(g)\mid g\in G\}$
   \; to  $D$ 
   \begin{description}[topsep=0pt,parsep=1pt]
   \item[2.1] foreach $g\in G$ of degree $d$
     \begin{description}[topsep=0pt,parsep=1pt]
     \item[2.1.1] Compute $h = \sarem(g,\; \text{SB})$
     \item[2.1.2] if $h \ne 0$ then \\
       redefine SB  as  SB $\cup\{h\}$\\
       redefine  MinGens as  MinGens $\cup\{h\}$
     \end{description}
   \item[2.2] if $d = D$ then \define{return MinGens}
   \item[2.3] compute $\{H_1, \dots, H_{t}\}$,
     the generators of degree $d+1$\\  of
     $\Rel(\LT_\sigma(g'_1),\dots,\LT_\sigma(g'_{s}))$, 
     where $g'_1,\dots,g'_s$ are the elements in SB
   \item[2.4] for $j=1,\dots,t$, compute
     $h_j = \sarem(H_j(g'_1, \dots,  g'_{s}),\; \text{SB})$
   \item[2.5] redefine SB as SB $\cup \{h_{1},\dots,h_{t}\}$
   \item[2.6] interreduce SB
   \end{description}
 \item[Output ] MinGens, a minimal system of generators of $S$.
\end{description}

\begin{proof}
Each iteration of the main loop has a fixed $d$
and  computes SB, a truncated \sagbi of
$K[G_{\le d}]$, where $G_{\le d} = \{ g_i \in G \mid \deg(g_i)\le
d\}$:
in Step 2.1 it is truncated to degree $d$,
 and in Steps 2.3-2.6 it is truncated to degree $d{+}1$, because it
 involves the relations up to degree $d{+}1$.
Having done that, in Step 2.1.1 of the next iteration,
we use SB to determine whether
each generator of degree $d{+}1$ is in $K[G_{\le d}]$, and also if
there is a, necessarily linear, relation with the previously added
generators of the same degree.

This procedure terminates because each iteration is finite, and there are
at most $D$ iterations.
\end{proof}

In the following example we see the algorithm at work.

\begin{example}\label{ex:satinterrednotenough-cont}
We reconsider Example~\ref{ex:satinterrednotenough}.
The algebra $S$ is standard graded and its $\sigma$-\satsagbi is
 $\{a_0,\, g_2,\, g_3,\, g_4, \, g_5\}$
where 
$g_5 = a_2^6  -8 \!\: a_0 \!\: a_1^3 a_2^2  -6 \!\: a_0 \!\: a_1^2 a_2^3  +3 \!\: a_0 \!\: a_1 \!\: a_2^4  +6 \!\: a_0^2 a_1 \!\: a_2^3  +4 \!\: a_0^2 a_2^4  -6 \!\: a_0^3 a_1^2 a_2 -12 \!\: a_0^3 a_1 \!\: a_2^2  +12 \!\: a_0^3 a_2^3  -a_0^4 a_2^2  -9 \!\: a_0^5 a_1 +6 \!\: a_0^5 a_2$.

Using Algorithm~\ref{alg:SubalgebraMinGens}  
we get $ \sat{S}{a_0} = K[a_0, g_2,g_3,g_4]$, and
indeed we can check that 
$$g_5 = 6a_0^4g_2 -3a_0^4g_3 -6a_0^2g_2g_3 -3a_0^2g_3^2 +4a_0^3g_4 -3g_2^2g_3 -g_3^3 -6a_0g_2g_4 +3a_0g_3g_4 +g_4^2$$

\end{example}

\section{The Graded Case: \sagbi for Saturation}
\label{sec:The Graded Case - Saturation}

We know that the main obstacle to the efficiency 
of Algorithm~\ref{alg:SubalgebraSaturation} is Step 2.2 which requires 
the computation of elimination ideals as explained in
Proposition~\ref{prop:subalgebraRepr}.
The first observation is that if the input polynomials in Step 2.2 are 
homogeneous, then it is well-known that the efficiency of the
computation of the elimination ideal can be improved.

The second observation is related to  a good use of 
the reduction described in Section~\ref{sec:Subalgebra-reduction}.
In general, it is desirable to streamline $\pig(f)$ as much as possible to
simplify the elimination process.
On the other end,
if the leading term of a polynomial $g$ is 
divisible by~$a_0$, then an \Srem of a polynomial~$f$ divided by
 $g$ in general does not ``simplify''~$\pig(f)$

Consequently, to maximize the chance of getting {\Srem}s 
divisible by $a_0$, our strategy is to use a term ordering $\sigma$ 
with the property that $\LTs(g) = \LTs(\pig(g))$ for every $g \in P{\setminus} \{0\}$.
These considerations motivate the following definition.

\begin{definition}\label{def:degrevtype}
Let $P=K[a_0, a_1, \dots, a_n]$, let $W \in \Mat_{m,n+1}(\ZZ)$ be a positive matrix,
Then let $\deg_W$ be the positive grading  on~$P$ defined by $W$.
A term ordering $\sigma$ on $\TT^{n+1}$ is said to be 
of \define{$a_0$-Deg$_W$Rev type} 
(or simply \define{$a_0$-DegRev type})
 is $\sigma$ is compatible with $\deg_W$ and
if $t, t' \in \TT^{n+1}$ are such that $\deg_W(t) = \deg_W(t')$ and 
$\log_{a_0}(t)< \log_{a_0}(t')$ then $t>_\sigma t'$.
\end{definition}

We recall that a way to construct a term ordering 
$\sigma$ of $a_0$-DegRev type
is to add to the matrix $W$
 the row $(-1,0,\dots, 0)$, and then completing it to a non-singular matrix.
For further details about this notion see~\cite[Sections 4.2 and  4.4]{KR2}.

Now we come to the main point of this section.
The most important feature of a positively graded
finitely generated $K$-subalgebra $S$ of $P$ which contains an indeterminate,
say $a_0$, is that the computation of $\sat{S}{a_0},$ and hence of $\Sat{a_0}(S)$
by Proposition~\ref{prop:MainProblem}.(b), can be essentially done by computing 
a suitable \sagbi of~$S$. Let us explain how.

Here is the main result of this section.

\begin{theorem}\label{thm:SatandSAGBI}
Let $\deg_W$ be the grading on~$P$ defined by a positive matrix~$W$,
and let~$\sigma$ on $\TT^{n+1}$ be a term ordering of 
$a_0$-{\rm DegRev} type.  Then let $S$ be a finitely generated \hbox{$W$-graded}
$K$-subalgebra of $P$,  let  $a_0\in S$, and
let $\SB$  be a $\sigma$-\sagbi of $S$.
Then the set 
$\{a_0\} \cup \{ \sat{g}{a_0} \mid g \in {\rm SB}\}$
is a $\sigma$-\sagbi of $\Sat{a_0}(S)$.
\end{theorem}

\begin{proof} 
It is enough to show that if $f \in \Sat{a_0}(S)$
is not divisible by $a_0$,
then $\LT_\sigma(f)$ is a power-product of elements in
$\{\LT_\sigma(\sat{g}{a_0}) \mid g \in {\rm SB}\}$.
From Proposition~\ref{prop:MainProblem}.(\ref{satequal})
we have  $\Sat{a_0}(S) = \sat{S}{a_0}$,
thus  $a_0^df \in S$ for some $d\in \NN$.
Therefore, there exist 
$\alpha_1, \dots, \alpha_t \in \NN$ and $g_1, \dots, g_t \in \SB$ such that 
$\LT_\sigma(a_0^d f) =  (\LT_\sigma (g_1))^{\alpha_1} \cdots 
(\LT_\sigma (g_t))^{\alpha_t} $. 
The assumptions on $S$ and $\sigma$ imply
that $a_0 \nmid \LT_\sigma(f)$, 
and for $i=1,\dots, t$, we have that 
$a_0 \nmid \LT_\sigma(\sat{g_i}{a_0})$
and there exists $r_i\in \NN$ such that 
$\LT_\sigma (g_i) = a_0^{r_i} \LT_\sigma(\sat{g_i}{a_0})$.
Thus, we have the equality 
$$
a_0^d \LT_\sigma(f) = 
a_0^{r_1\alpha_1} (\LT_\sigma (\sat{g_1}{a_0}))^{\alpha_1} 
\; \cdots \;
a_0^{r_t\alpha_t} (\LT_\sigma (\sat{g_t}{a_0}))^{\alpha_t}
$$
By setting $a_0=1$ we get the desired conclusion.
\end{proof}


The following easy example shows that the assumption about the term 
ordering $\sigma$ in the above theorem is essential.

\begin{example}\label{ex:essentialsigma}
Let $P=\QQ[a_0,a_1,a_2,a_3]$ and let $S=\QQ[\g,g_2,g_3]$
where we have ${\g = a_0}, \allowbreak\, g_2=a_0a_2 -a_1^2,\,  g_3= a_0a_3^2 -a_1^3$.
If $\sigma=\tt DegLex$, which is
not of $a_0$-DegRev type,
then $\LT_\sigma(a_0) = a_0$,
$\LT_\sigma(g_2) = a_0a_2$, and $\LT_\sigma(g_3) = a_0a_3^2$.
The three power products are algebraically independent, hence $\SB = \{a_0, g_2,g_3\}$
is a $\sigma$-\sagbi of $S$ by~\cite[Proposition 6.6.11]{KR2}.
Instead, if $\sigma $ is the term ordering defined by the matrix 
$\Big( \begin{smallmatrix}\ \ 1 &\ \  1 &\ 1\cr
-1 &\ \  0 &\ 0\cr
\ \ 0 & -1 &\ 0
\end{smallmatrix} \Big)$,
then $S$ is $W$-graded where $W = (1\ 1\ 1)$, and $\sigma$ is
a term ordering of $a_0$-DegRev type. Now we have 
$\LT_\sigma(a_0) = a_0$,
$\LT_\sigma(g_2) = a_1^2$, and $\LT_\sigma(g_3) = a_1^3$.
The $\sigma$-\sagbi of $S$ is 
SB = $\{a_0, g_2, g_3, g_4\}$ where 
$g_4 =   a_0a_1^4a_2 -{\tfrac{2}{3}}_{\mathstrut}a_0a_1^3a_3^2 -a_0^2a_1^2a_2^2 
+ \tfrac{1}{3}a_0^2a_3^4 +\tfrac{1}{3}a_0^3a_2^3 $.
By saturating $g_4$ we get $\tilde{g}_4 = \sat{g_4}{a_0} = 
a_1^4a_2 -\tfrac{2}{3}a_1^3a_3^2 -a_0a_1^2a_2^2 
+ \tfrac{1}{3}a_0a_3^4 +\tfrac{1}{3}a_0^2a_2^3$.
Moreover,
 by Algorithm~\ref{alg:SubalgebraMinGens} we check that 
$\Sat{a_0}(S) $ is minimally generated by $(a_0,g_2,g_3,\tilde{g}_4)$.
\end{example}

The following example illustrates a subtlety of the theorem. 
It happens that while the SAGBI basis of $S$ is infinite, the SAGBI basis of $\Sat{a_0}(S)$
is finite.

\begin{example}\label{ex:infinite-finite}
Let $P=\QQ[a_0,a_1,a_2]$ be graded by the matrix
$W = \left(\begin{smallmatrix} 
\ \ 1& 1& 1\\ 
-1& 0& 0  \end{smallmatrix} \right)$ and let $S=\QQ[\g,g_2,g_3, g_4,g_5]$
where we have $\g = a_0,\,  g_2= a_0a_1, \, g_3=a_1+a_2,\,  g_4 = a_1a_2,\, \allowbreak\, g_5 = a_1a_2^2$.
If $\sigma$ is a term ordering compatible with $W$, the $\sigma$-SAGBI
basis of $S$ is not finite. It is 
$$ 
\{a_0,\ a_1+a_2,\ a_0a_2,\ a_1a_2, \ a_0a_2^2, \ a_1a_2^2,\dots, a_0a_2^i, \ a_1a_2^i,\dots \}
$$
while the $\sigma$-SAGBI basis of $\Sat{a_0}(S)$ is finite. It is 
$$
\{ a_0, \ a_1+a_2, \ a_2 \}.
$$
\end{example}

The following procedure combines the saturation of the elements of a
$\sigma$-\sagbi,
as described in Theorem~\ref{thm:SatandSAGBI}, within the iterations of
the \sagbi computation.
It is a procedure because termination is not guaranteed, but if it
  terminates the output is correct.

\MyProcedure{SatSAGBI}  

\begin{description}[topsep=0pt,parsep=1pt] 
 \item[\textit{notation:}]  $P =K[a_0,\dots,a_n]$ is a polynomial
   ring graded by a positive matrix,
and let~$\sigma$ on $\TT^{n+1}$ be a term ordering of 
$a_0$-DegRev type. 
 \item[Input]  
$S = K[g_1,\dots,g_r]\subseteq P$, 
with $g_1,\dots, g_r$ homogeneous.

\item[1] Let $G = \{g_1,\dots,g_r\}$
 \item[2] \textit{Main Loop:}
   \begin{description}[topsep=0pt,parsep=1pt]
   \item[2.1] compute $G' = \{g'_1,\dots,g'_{s}\}$ the sat-interreduction
     of $G$.
   \item[2.2] compute $\{H_1, \dots, H_{t}\}$,
     a set of generators of
     $\Rel(\LT_\sigma(g'_1),\dots,\LT_\sigma(g'_{s}))$
   \item[2.3] for $j=1,\dots,t$, let 
     $h_j = \sat{\sarem(H_j(g'_1, \dots,  g'_{s}),\; G')}{a_0}$
   \item[2.4] if $h_1=\dots=h_{t} = 0$
          then \define{return $\{a_0\}\cup G'$} 
   \item[2.5] redefine $G$ as $G' \cup \{h_{1},\dots,h_{t}\}$
   \end{description}
 \item[Output ] $\{a_0\}\cup G'$ , a  $\sigma$-\sagbi of $\sat{S}{a_0}$
\end{description}

\begin{proof}
The definition of $G'$ in Step~2.1, and 
redefinition of $G$ in Step~2.5 correspond to the
definition of new subalgebras 
$S' = K[G']$
and 
$S'' = K[G' \cup \{h_{1},\dots,h_{t}\}]$
which satisfy $S \subseteq S' \subseteq S'' \subseteq \sat{S}{a_0}$,
thus all algebras defined in this procedure
have saturation $\sat{S}{a_0}$, by Theorem~\ref{thm:inclusion}.(c).

Each iteration of Step~2.1 is 
equivalent to restarting the computation of a
$\sigma$-\sagbi of $K[G']$, 
where all the elements in $G'$ are $a_0$-saturated.

If the procedure stops in Step~2.4, then 
$\{a_0\}\cup G'$ is a  $\sigma$-\sagbi of the algebra
 $A = K[\{a_0\}\cup G']$
and therefore, by Theorem~\ref{thm:SatandSAGBI}, $A = \sat{A}{a_0}$.

In conclusion, if it terminates, the output 
is the $\sigma$-\sagbi of $\sat{S}{a_0}$.
\end{proof}

Is this procedure the definitive solution of our problem?
The answer is yes and no.
The following example provides a negative answer
by showing that for some input this procedure cannot terminate because
there is no finite \sagbi.

\begin{example}\label{ex:infiniteSAGBI}
Let $P=\QQ[a_0,a_1,a_2]$ and let $S=\QQ[\g,g_2,g_3,g_4]$
where we have $\g = a_0$, \ $g_2=a_1{+}a_2$,\  $g_3= a_1a_2$,\  $g_4 = a_1a_2^2$.
Since $g_2,g_3,g_4$ do not involve $a_0$, it is clear that $S= \sat{S}{a_0}$.
On the other hand, whatever term ordering we choose, the \sagbi and the \satsagbi of $S$
are infinite (see~\cite[Example 6.6.7]{KR2}).
\end{example}

However, the computation of Example~\ref{ex:infinite-finite}
immediately terminates when we add the generator $a_2 = \sat{a_0a_2}{a_0}$,
and also terminates for many other examples we computed, leading us to
formulate the following conjecture.

\goodbreak
\begin{conjecture}\label{conj:SatSAGBI}
If there is a finite $\sigma$-\sagbi of~$\sat{S}{a_0}$, 
Procedure~\ref{alg:SatSAGBI} terminates in a finite number of
iterations, hence it is an algorithm.
\end{conjecture}

\begin{remark}
The delicate point in proving 
this conjecture is that Steps~2.1 and~2.5
might produce a sequence of algebras
ever closer to $\Sat{g}(S)$, but never getting to it.

\end{remark}

The positive side is that the computation using \sagbis
provides not only a set of generators of the saturation of $S$ but also
a \sagbi of it. Secondly the computation of a \sagbi
needs to determine relations only among power-products, thus may
use toric ideals whose computation is considerably faster than the computation 
via general elimination needed for determining $\Relg$.

Let us show an example where the  above procedure works very well.

\begin{example}\label{ex:standardgraded}
We let $P=\QQ[a_0,a_1,a_2]$ graded by the matrix 
$W= (1,1,1)$ and  use a term ordering of $a_0$-DegRev type. Then let 
$\g = a_0$,\,  $g_2 = a_1^2 -a_2^2 +a_0a_2$, \,
$g_3 = a_1a_2 -a_2^2 +a_0a_1$,\,
$g_4 = a_1^3$,\,
$g_5 = a_2^4$ . We want to saturate the algebra 
$S = \QQ[\g, g_2, g_3, g_4, g_5]$
with respect to $g$. Using Algorithm~\ref{alg:SatSAGBI}, 
we get a \sagbi of $\Sat{\g}(S)$ which consists of 12 polynomials.
Using Algorithm~\ref{alg:SubalgebraMinGens}
 we  get a minimal set of generators of  $\Sat{\g}(S)$.
The result is $\Sat{\g}(S) = \QQ[\g, g_2, g_3, g_4, g_5, g_6, g_7, g_8]$ where 
{\small
$$
\begin{array}{lll}
g_6&=& a_1^3 a_2^2  -\frac{23}{15} \!\: a_1^2 a_2^3  -\frac{11}{45} \!\: a_1 \!\: a_2^4  +\frac{44}{45} \!\: a_2^5  -\frac{5}{18} \!\: a_0 \!\: a_1 \!\: a_2^3  +\frac{6}{5} \!\: a_0^2 a_1^2 a_2 -\frac{23}{30} \!\: a_0^2 a_1 \!\: a_2^2  +\frac{5}{6} \!\: a_0^2 a_2^3  -\frac{1}{5} \!\: a_0^3 a_2^2 \\
& & +\frac{11}{15} \!\: a_0^4 a_1 -\frac{1}{2} \!\: a_0^4 a_2 \\

g_7&=& a_2^7  -\frac{295}{2} \!\: a_0^2 a_1^2 a_2^3  -\frac{65}{6} \!\: a_0^2 a_1 \!\: a_2^4  +\frac{119}{6} \!\: a_0^2 a_2^5  +\frac{1217}{12} \!\: a_0^3 a_1 \!\: a_2^3  -30 \!\: a_0^4 a_1^2 a_2 +\frac{319}{4} \!\: a_0^4 a_1 \!\: a_2^2  -\frac{275}{4} \!\: a_0^4 a_2^3  \\
 & & -42 \!\: a_0^5 a_2^2  +\frac{65}{2} \!\: a_0^6 a_1 +\frac{219}{4} \!\: a_0^6 a_2 \\

g_8&=& a_1 \!\: a_2^6  -\frac{576}{5} \!\: a_0^2 a_1^2 a_2^3  -\frac{179}{30} \!\: a_0^2 a_1 \!\: a_2^4  +\frac{193}{15} \!\: a_0^2 a_2^5  +\frac{214}{3} \!\: a_0^3 a_1 \!\: a_2^3  -\frac{54}{5} \!\: a_0^4 a_1^2 a_2 +\frac{262}{5} \!\: a_0^4 a_1 \!\: a_2^2  \\
 & & -60 \!\: a_0^4 a_2^3  -\frac{126}{5} \!\: a_0^5 a_2^2  +\frac{239}{10} \!\: a_0^6 a_1 +39 \!\: a_0^6 a_2
\end{array}
$$
}
In this case the computation takes a few seconds. We could also use
Algorithm~\ref{alg:SubalgebraSaturation}
to compute $\Sat{g}(S)$, but generally it gives neither a minimal set
  of generators, nor a \sagbi of it.
\end{example}

In the following example the performance of
Procedure~\ref{alg:SatSAGBI} is far superior.

\begin{example}\label{ex:standardgradedMedium}
We let $P=\QQ[a_0,a_1,a_2,a_3]$ graded by the matrix 
$W= (1,1,1,1)$ and  use a term ordering of $a_0$-DegRev type. Then let 
$g = a_0$,\,  $g_2 = a_1^2 -a_0a_3$, \,
$g_3 = a_1a_2  +a_0a_1$,\,
$g_4 = a_3^2$,\,
$g_5 = a_2^2$,\, 
$g_6 = a_1^3 -a_2^3$.
We want to saturate the algebra 
$S = \QQ[\g, g_2, g_3, g_4, g_5,g_6]$ with respect to $g$. 
Using Procedure~\ref{alg:SatSAGBI},  we get a \sagbi of $\Sat{g}(S)$ which 
consists of 21 polynomials.
Using Algorithm~\ref{alg:SubalgebraMinGens}
 we  get a minimal set of generators for $\Sat{\g}(S)$.
The result is $\Sat{\g}(S) = \QQ[\g,g_2,g_3,g_4,g_5,g_6, g_7,\dots,  g_{15}]$ where 
{\small
$$
\begin{array}{lll}
g_7 & = & a_1^2 a_2 +\frac{1}{2} \!\: a_2^2 a_3 +\frac{1}{2} \!\: a_0^2 a_3 \\

g_8 & = &  a_1^4 a_3 -a_1 \!\: a_2^3 a_3 -a_0 \!\: a_1 \!\: a_2^2 a_3 +\frac{2}{3} \!\: a_0^2 a_2^3  -a_0^2 a_1 \!\: a_2 \!\: a_3 -\frac{2}{3} \!\: a_0^2 a_3^3  -a_0^3 a_1 \!\: a_3 \\

g_9 & = &  a_2^3 a_3^3  -a_0 \!\: a_1 \!\: a_2^3 a_3 +\frac{1}{4} \!\: a_0 \!\: a_2^2 a_3^3  -a_0^2 a_1 \!\: a_2^2 a_3 +\frac{1}{2} \!\: a_0^3 a_2^3  -\frac{3}{4} \!\: a_0^3 a_1 \!\: a_2 \!\: a_3 +\frac{3}{4} \!\: a_0^3 a_3^3  -\frac{3}{4} \!\: a_0^4 a_1 \!\: a_3 \\

g_{10} & = &  a_1 \!\: a_2^4 a_3 +a_0 \!\: a_1 \!\: a_2^3 a_3 \\

g_{11} & = &  a_1^4 a_2 \!\: a_3 +\frac{4}{3} \!\: a_0 \!\: a_1 \!\: a_2^3 a_3 +\frac{2}{3} \!\: a_0 \!\: a_2^2 a_3^3  +\frac{4}{3} \!\: a_0^2 a_1 \!\: a_2^2 a_3 -\frac{2}{3} \!\: a_0^3 a_2^3  +a_0^3 a_1 \!\: a_2 \!\: a_3 +a_0^4 a_1 \!\: a_3 \\

g_{12} & = &  a_1^2 a_2^2 a_3^3  +\frac{1}{2} \!\: a_2^3 a_3^4  -\frac{1}{4} \!\: a_0 \!\: a_2^5 a_3 +\frac{1}{4} \!\: a_0 \!\: a_1^2 a_2 \!\: a_3^3 
 +\frac{9}{8} \!\: a_0^2 a_1^3 a_2 \!\: a_3 -\frac{1}{8} \!\: a_0^2 a_2^4 a_3 -\frac{3}{4} \!\: a_0^2 a_1^2 a_3^3 \\
 & &  +\frac{9}{8} \!\: a_0^3 a_1^3 a_3 +\frac{1}{2} \!\: a_0^3 a_2^3 a_3 +\frac{3}{8} \!\: a_0^4 a_2^2 a_3 \\

g_{13} & = &  a_2^6 a_3 +\frac{1}{8} \!\: a_2^3 a_3^4  +\frac{39}{16} \!\: a_0 \!\: a_2^5 a_3 +\frac{9}{16} \!\: a_0 \!\: a_1^2 a_2 \!\: a_3^3  -\frac{135}{32} \!\: a_0^2 a_1^3 a_2 \!\: a_3 +\frac{15}{32} \!\: a_0^2 a_2^4 a_3 +\frac{9}{16} \!\: a_0^2 a_1^2 a_3^3\\
& &   -\frac{135}{32} \!\: a_0^3 a_1^3 a_3 -\frac{19}{8} \!\: a_0^3 a_2^3 a_3 -\frac{45}{32} \!\: a_0^4 a_2^2 a_3 \\

g_{14} & = & a_1 \!\: a_2^5 a_3 -\frac{1}{4} \!\: a_2^4 a_3^3  -a_0^2 a_1 \!\: a_2^3 a_3 -\frac{1}{2} \!\: a_0^2 a_2^2 a_3^3  +\frac{3}{4} \!\: a_0^4 a_3^3 \\

g_{15} & = &  a_1^2 a_2^4 a_3 -\frac{1}{4} \!\: a_1 \!\: a_2^3 a_3^3  -\frac{1}{4} \!\: a_0 \!\: a_1 \!\: a_2^2 a_3^3  -2 \!\: a_0^2 a_1^2 a_2^2 a_3 -\frac{3}{4} \!\: a_0^2 a_1 \!\: a_2 \!\: a_3^3  -a_0^3 a_1^2 a_2 \!\: a_3 -\frac{3}{4} \!\: a_0^3 a_1 \!\: a_3^3 
\end{array}
$$
}
The computation took about 75 seconds using 
Procedure~\ref{alg:SatSAGBI} and Algorithm~\ref{alg:SubalgebraMinGens}.
We tried to do the computation using Algorithm~\ref{alg:SubalgebraSaturation}
and we did not succeed.
\end{example}

\section{A Special Multigraded Case: Truncated \sagbi for Saturation}
\label{sec:U-invariants}
In general, the computation of $\sat{S}{a_0}$ is very expensive.
The performance of Algorithm~\ref{alg:SubalgebraSaturation} is poor 
even for examples of moderate size. The performance of
Procedure~\ref{alg:SatSAGBI} 
is usually
much better, but the computation of a \sagbi may be prohibitive as well.
However, there is a situation where it is possible to compute $(\sat{S}{a_0})_{\le d}$,
in other words a truncation of $\sat{S}{a_0}$ at  degree $d$. Let us see how. 

In Section~\ref{sec:The Graded Case - Saturation} we have already seen  
that the main requirement to compute the saturation of $S$ with respect to an indeterminate, 
is to compute an $a_0$-saturated \sagbi of $S$ with respect to 
a term ordering of $a_0$-DegRev type.
Our question is: if the computation of a saturating \sagbi is prohibitive, can we at least compute a 
truncation of a saturating \sagbi  at  a given degree? The main obstacle is that when we saturate
a computed polynomial, we may lower its degree. If $a_0$ is the chosen indeterminate,
the only possibility of keeping the degree fixed
is 
when the input is homogeneous with respect to
a grading where $\deg(a_0) = 0$. This condition is clearly incompatible with a term ordering
of $a_0$-DegRev type, unless 
the input is homogeneous also with respect to
 another grading with $\deg(a_0)>0$.

The following example  shows that  in many cases 
the computation of the saturation may be too hard
even when working over a small prime field.

\begin{example}\label{ex:bigradedHard}
We let $P=\ZZ/(101)[a_0,a_1,a_2, a_3, a_4]$ standard graded by 
the matrix  $W= (1,1,1,1,1)$, and use a term ordering of $a_0$-DegRev type. 
Then we let 
$\g=a_0$,\,
$g_2= a_1^2-a_2^2+a_0a_3$,\,
$g_3= a_1^3 +a_2^3 +a_0^2a_4$,\,
$g_4= a_3^3- a_0a_4^2$,\,
$g_5= a_4^3$, and
want to saturate the algebra 
$S = \ZZ/(101)[\g, g_2, g_3, g_4, g_5]$
with respect to $\g$. 

No matter which algorithm we use, there is no way.
However, we observe that the given polynomials are also homogeneous
with respect to the grading given by $(0\  1\  1\  2\  3)$ and this observation 
suggests an interesting approach which we are going to explain.
We continue this discussion in Example~\ref{ex:bigradedHard-cont}.
\end{example}

We start with the following easy lemma.

\begin{lemma}\label{lem:finitepowerproducts}
Let $P= K[a_0,a_1,\dots, a_n]$, let $d_1, \dots, d_n \in \NN_+$, 
let $P$ be (single) graded by  $W=(0\ d_1 \cdots\,  d_n)$, 
and let $S \subset P$ be a finitely generated monomial $K$-algebra.
Then let $d \in \NN_+$, and let 
$S_d = \{f \in P \mid f \hbox{\ \rm homogeneous of degree\ } $d$\}$.  
\begin{enumerate}
\item The set $S_d$ is a $K[a_0]$-module.

\item The $K[a_0]$-module $S_d$ is finitely generated, and there is a unique 
set of power products which minimally generate it.
\end{enumerate}
\end{lemma}

\begin{proof}
Claim (a) follows from the fact that $\deg(a_0) = 0$. Let $S\subset P_d$ denote the set of 
power products of degree $d$ in $\TT(a_1, \dots, a_n)$, and let $t_1, \dots, t_r$ be 
the unique basis of power products of $S_d$ as a $K$-vector space. For each $t_i$ there is 
a minimum exponent $e_i$ such that $\tau_i =a_0^{e_i}t_i \in S_d$. It follows that $S_d$ is
 minimally generated by~$\{\tau_1, \dots, \tau_r\}$.
\end{proof}

\begin{proposition}\label{prop:truncsatsagbi}
Let $P = K[{a_0}_{\mathstrut}, a_1,\dots, a_n]$, 
let $d_1,\dots, d_n \in \NN_+$, $p_1, \dots, p_n \in \ZZ$, let 
$P$ be graded by  $W$ whose first two rows are 
$W_1 = (0 \    d_1  \cdots\,  d_n)$,  $W_2 = (1 \    p_1  \cdots\,  p_n)$,
and let~$\sigma$ be a term 
ordering on $\TT^{n+1}$ compatible with $W$ and  of $a_0$-DegRev type.
Then let~$S$ be a finitely generated \hbox{$W$-graded}
$K$-subalgebra of $P$,  let  $a_0\in S$, let $\SB$
be a $\sigma$-\sagbi of~$S$,  let $d \in \NN_+$, and
let $\SB_{\le d} = \{g \in \SB \mid  g\hbox{\ is $W$-homogeneous and\ }  \deg_{W_1}(g) \le d\}$.

Then 
$\{a_0\} \cup \{ \sat{g}{a_0} \mid g \in {\rm SB_{\le d}}\}$  is 
a $d$-truncated $\sigma$-\sagbi of $S$.
\end{proposition}

\begin{proof}
Our assumptions are compatible with those of Theorem~\ref{thm:SatandSAGBI}.
When we compute a $\sigma$-\sagbi of $S$ we may proceed by increasing degrees
as suggested by Remark~\ref{rem:IncreasingDegreeSAGBI}. We proceed using the degree $\deg_{W_1}$.
The merit  is that the saturation of a polynomial does not change $\deg_{W_1}$.
Then Lemma~\ref{lem:finitepowerproducts} shows that the computation of the $\sigma$-\sagbi
jumps over $d$ and clearly it does not come back anymore. The conclusion follows.
\end{proof}

\MyAlgorithm{TruncSatSAGBI}
 
\begin{description}[topsep=0pt,parsep=1pt] 
 \item[\textit{notation:}]  $P =K[a_0,\dots,a_n]$ is a polynomial
   ring graded by  $W$ whose first two rows are 

$W_1 = (0 \    d_1  \cdots\,  d_n)$,  $W_2 = (1 \    p_1  \cdots\,  p_n)$,
with $d_1,\dots, d_n \in \NN_+$, $p_1, \dots, p_n \in \ZZ$.

Let~$\sigma$ be a term ordering on $\TT^{n+1}$  
compatible with $W$ and  of $a_0$-DegRev type.
 \item[Input]  
$S = K[g_1,\dots,g_r]\subseteq P$, with $g_i$ $W$-homogeneous for $i=1,\dots,r$.
\item[1] Let $G = \{g_1,\dots,g_r\}$
 \item[2] \textit{Main Loop:}
   \begin{description}[topsep=0pt,parsep=1pt]
   \item[2.1] compute $G' = \{g'_1,\dots,g'_{s}\}$ the sat-interreduction
     of $G$.
   \item[2.2] compute $\{H_1, \dots, H_{t}\}$,
     the subset of elements  of $W_1$-degree $\le d$ in a set of generators of
     $\Rel(\LT_\sigma(g'_1),\dots,\LT_\sigma(g'_{s}))$
   \item[2.3] for $j=1,\dots,t$, let 
     $h_j = \sat{\sarem(H_j(g'_1, \dots,  g'_{s}),\; G')}{a_0}$
   \item[2.4] if $h_1=\dots=h_{t} = 0$
          then \define{return $G'$}
   \item[2.5] redefine $G$ as $G' \cup \{h_{1},\dots,h_{t}\}$
   \end{description}
 \item[Output ] $G'$, a  $\sigma$-\sagbi of $\sat{S}{a_0}$ truncated at
   $W_1$-degree $d$.
\end{description}

\begin{proof}
Correctness  and termination follow immediately from Proposition~\ref{prop:truncsatsagbi}.
\end{proof}

Let us go back to Example~\ref{ex:bigradedHard}.

\begin{example}\label{ex:bigradedHard-cont}
Using the data introduced in Example~\ref{ex:bigradedHard} we compute
$(\sat{S}{a_0})_{\le 30}$. 

In less than a second we get 
$(\sat{S}{a_0})_{\le 30} = (K[a_0,g_2,\dots, g_6, g_7])_{\le 30}$ where
$g_7$ is a polynomial of bi-degree $(30,29)$ with 
767 terms and
$\LT(g_7) = a_1^{20}a_2^8a_3$.

In about 23 seconds we compute 
$(\sat{S}{a_0})_{\le 90} = (K[a_0,g_2,\dots, g_6, g_7,g_8])_{\le 90}$ where
$g_8$  is a polynomial of bi-degree $(90,86)$ with 
19559 terms and
$\LT(g_8) = a_1^59a_2^24a_3^2a_4$.

Then we try to compute $(\sat{S}{a_0})_{\le 300}$ and after about 30 minutes we realise that 
the algorithm gets a new polynomial
with 516775 terms and
 leading term  $a_1^{176} a_2^{73} a_3^6 a_4^3$.
At this point  we understand that
the computation is not going to end
in a reasonable amount of time. 

In conclusion, we are able to compute $(\sat{S}{a_0})_{\le 90}$, but we are not even able to 
know whether $\sat{S}{a_0}$ is finitely generated or not.
\end{example}

\subsection{Computing $U$-invariants}

Unlike Example~\ref{ex:bigradedHard-cont},  there are cases where a bit of extra knowledge
allows us to fully compute the saturation of a subalgebra using the technique of truncation.
And we go back to the introduction where we started
our discussion about the computation 
the classical $U$-invariants, which gave us
 a first motivation of our work.  Recall
 that the problem is to compute the $\CC$-subalgebra 
 $S_n=\CC[c_2,\dots, c_n][a_0,a_0^{-1}]\cap \CC[a_0,\dots, a_n]$ of
 $ \CC[a_0,\dots, a_n]$ where the polynomials
 $c_i$'s are defined in the introduction.
 
 First of all, it follows from Proposition~\ref{prop:MainProblem}.(\ref{satequal}) that 
 $S_n= \CC[a_0, c_2,\dots, c_n] : a_0^\infty$.
 Then we observe  that $a_0, c_2,\dots, c_n$ are elements of  
 the polynomial ring $P=\QQ[a_0, a_1, \dots,,a_n]$, hence all the 
 computation of the \satsagbi involves polynomials in $P$,
 so the generators of  $S_n$ lie in~$P$.
 To see more on this topic see~\cite{RS}.
 
 The third remark is that 
$a_0, c_2,\dots, c_n$ are bi-homogeneous elements in $P$ graded by
the positive matrix 
 $W_n = \left(\begin{smallmatrix} 0 & 1 & \cdots & n\cr 1 & 1 &  \cdots & 1  
 \end{smallmatrix}  \right)$.

 Finally, classical results show that $S_n$
 is finitely generated and, for some~$n$, 
 compute the bi-degrees of a minimal set of generators.
 Consequently, according to Proposition~\ref{prop:truncsatsagbi} we
 can compute $S_n$
 by truncating the \satsagbi at the  maximum weighted degree
 given by the grading $(0 \ \ 1 \  \cdots\  n)$,
the first row of $W_n$.
And this is what we are able to do for the easy cases
 $S_3$ and $S_4$ and for the non-trivial cases~$S_5$ and $S_6$. 
 Our results agree with the classical ones (see~\cite{pop}). 
 Our main contribution is that we are able to directly compute the invariants.

\begin{example}\label{ex:U3}
In a split second the computation of $S_3$ yields the following result.
We~have $S_3  = \CC[a_0, 2g_2, 3g_3, g_4]$ where 
$g_4 = a_1^2a_2^2 -2a_1^3a_3 -\tfrac{8}{3}a_0a_2^3 +6a_0a_1a_2a_3 -3a_0^2a_3^2$.
\end{example}

\begin{example}\label{ex:U4}
In a split second the computation of $S_4$  yields the following result.

We have $S_4= \CC[a_0, 2g_2, 3g_3, g_4,g_5]$ where 
$g_4 = a_2^2 -2a_1a_3 +2a_0a_4$ and 

${g_5 = a_2^3 -3a_1a_2a_3 +3a_1^2a_4 +\tfrac{9}{2}a_0a_3^2 -6a_0a_2a_4}$.
\end{example}

Here we come to the non-trivial cases.

\begin{example}\label{ex:U5}
It is known that the highest weighted degree of a generator in a set of 
minimal generators of $S_5$  is 45. Therefore we compute a 
$\sigma$-\satsagbi of $\QQ[a_0, c_2,\dots, c_5]$ 
truncated in weighted degree~45,
where $\sigma$ is a term ordering $a_0$-DegRev type compatible with~$W$.

We need about 7 minutes to compute a set of  57 generators of the truncated
\satsagbi and another 5 minutes to minimalize it. The conclusion is that
we get 23 generators. Their leading terms are 
$$
\begin{array}{l}
a_0, \,  a_1^2,\;  a_1^3,\;  a_2^2,\;  a_1a_2^2,\;  a_2^3,\;  a_1a_2^3,\;  a_1^2a_3^2,\;  a_1^3a_3^2,\; 
 a_2^2a_3^2,\;   a_1a_2^2a_3^2,\;  a_2^3a_3^2,\;  a_1a_2^3a_3^2,\;  a_1a_2^2a_3^3,\;  a_1^3a_3^4, \\ 
 a_2^4a_3^3 a_2^5a_3^3,\;  a_1^2a_2^2a_3^2a_4^2,\;  a_1^2a_2^2a_3^5,\;  a_1^2a_2^2a_3^7,\;  
 a_1^2a_2^2a_3^8,\;  a_1^2a_2^3a_3^8,\;  a_1^2a_2^5a_3^{11}
  \end{array}
$$ 
Their bi-degrees are
$$
\begin{array}{l}
(0,  1),\,  (2,  2),\,  (3,  3), (4,  2),\,  (5,  3),\,  (6,  3),\,  (7,  4),  (8,  4),\,  (9,  5),\,  (10,  4),  
(11,  5),\,  (12,  5),\, (13,  6),  \\
 (14,  6),\,  (15,  7),\,  (17,  7),\,  (19,  8),\,  (20,  8),\,  (21,  9),\,  (27,  11),\,  (30,  12),\,  (32,  13),\,  (45,  18)
 \end{array}
$$ 
The sizes of the supports of the 23 polynomials are
$$
(1,\;  2,\;  3,\;  3,\;  5,\;  5,\;  9,\;  9,\;  13,\;  12,\;  17,\;  20,\;  29,\;  30,\;  36,\;  49,\;  65,\;  
59,\;  93,\;  183,\;  247,\;  319,\;  848)
$$
For the interested reader
the following link provides the code we wrote and the actual
polynomials we computed \\
 \small{ \verb|http://www.dima.unige.it/~bigatti/data/ComputingSaturationsOfSubalgebras/|.}
 \end{example}
 
\begin{example}\label{ex:U6}
It is known that the highest weighted degree of a generator in a set of 
minimal generators of $S_6$ is 45. Therefore we compute a 
$\sigma$-\satsagbi of $\QQ[a_0, c_2,\dots, c_6]$ truncated in weighted degree~45,
where $\sigma$ is a term ordering $a_0$-DegRev type compatible with~$W$.

We need about  2 hours and 15  minutes to compute a set of  83   generators of the truncated
\satsagbi and another 1 hour and   40  minutes to minimalize it. The conclusion is that
we get 26 generators. Their leading terms are 
$$
\begin{array}{l}
 a_0,\; a_1^2,\; a_1^3,\, a_2^2,\; a_1a_2^2,\; a_3^2,\; a_2^3,\; a_1a_2^3,\; a_2a_3^2,\; a_1a_2a_3^2,\; a_1a_3^3,\;
 a_1^2a_3^3,\, a_2^2a_4^2,\, a_2^2a_3^3,\, a_2^3a_4^2,\, a_1a_2^3a_4^2,\\
 a_2^3a_3^3,\, a_2^3a_4^3,\,
 a_1a_2^2a_3^2a_4^2,\, a_2^4a_4^3,\, a_2^3a_3^3a_4^2,\,  a_1a_2^3a_3^2a_4^3,\; a_2^4a_3^3a_4^3,\,
a_1^2a_2^2a_3^2a_4^2a_5^2,\,
a_1^2a_2^2a_3^3a_4^5,\; a_1^2a_2^2a_3^5a_4^6
\end{array}  
$$
Their bi-degrees are
$$
\begin{array}{l}
(0,  1),\; (2,  2),\, (3,  3),\;  (4,  2),\;  (5,  3),\;  (6,  2),\;  (6,  3),\;  (7,  4), \; (8,  3),\;  (9,  4), \; (10,  4), \; (11,  5), \\
(12,  4), \; (13,  5),\;  (14,  5), \; (15,  6), \; (15,  6), \; (18,  6), \; (19,  7),\;  (20,  7), \; (23,  8),\,  (25,  9),\, (29,  10),\\  
(30,  10),\; (35,  12), \; (45,  15)
 \end{array}
$$ 
The sizes of the supports of the 26 polynomials are
$$
\begin{array}{l}
1,\;  2,\;  3,\;  3,\;  5,\;  4,\;  6,\;  9,\;  8,\;  13,\;  12,\;  20,\;  16,\;  28,\;  29,\;  42,\;  47,\;  52,\;  77,\;  85,\;  135,\;  196,\;  312,
\\  
246,\,  586,\,  1370
\end{array}
$$
As in Example~\ref{ex:U5}, the following link provides the code and the 
polynomials\\
 \small{ \verb|http://www.dima.unige.it/~bigatti/data/ComputingSaturationsOfSubalgebras/|.}
\end{example}

 \section{Conclusions}
In this paper we deal with the problem of saturating~$S$
with respect to~$g$, and we denote the resulting \hbox{$K$-algebra}  by~$\Sat{g}(S)$.
Here $S$ is a finitely generated $K$-subalgebra of a polynomial ring $P=K[a_0,a_1, \dots, a_n]$
with $K$ being any field.
It turns out that $\Sat{g}(S) = S:g^\infty$ if $g\in S$ which we
always assume throughout this paper.
After several preparatory results we get Algorithm~\ref{alg:SubalgebraSaturation} which solves the problem
if $\Sat{g}(S)$ is a finitely generated $K$-algebra. If not, the algorithm is simply a procedure which
allows us to get closer and closer to the saturation. 
As said in the introduction, ~\cite{DK2015} contains a similar result
(see~\cite[Semi-algorithm 4.10.16]{DK2015}).

Then we introduce techniques coming from the theory of \sagbis which show their power mainly
in the case that $S$ is graded. We describe an algorithm which allows to minimalize a given set of 
homogeneous generators of a $K$-subalgebra of $P$ (see Algorithm~\ref{alg:SubalgebraMinGens}).
Then Theorem~\ref{thm:SatandSAGBI} illustrates a nice interplay between saturating $S$ with
respect to an indeterminate and computing a special \sagbi of $S$. The first output of this theorem
is Procedure~\ref{alg:SatSAGBI} whose power is illustrated by some
interesting examples.
We prove that Procedure~\ref{alg:SatSAGBI} is correct and conjecture that it terminates
whenever $\Sat{g}(S)$ is a finitely generated $K$-algebra (see Conjecture~\ref{conj:SatSAGBI}).

The final part of the paper is dedicated to find a direct attack to the problem of computing 
the algebras $S_n$ of $U$-invariants, a classical problem which goes back 
to the nineteenth century.
We succeed up to degree 6,  we do it without the assumption that $K = \CC$, 
and we are able to compute not only a minimal set of $U$-invariants, 
but also a truncated \sagbi of the corresponding algebra.

If $g \notin S$ we denote $S[g^{-1}] \cap P$ by \textit{weak saturation} of $S$ with respect to $g$.
It turns out that this algebra is very different from $S:g^\infty$ which, in general, is not even an algebra.
The problem of computing  $S[g^{-1}] \cap P$ if $g \notin S$ can be inspiration for future research.

\bibliographystyle{plain}

\end{document}